\documentclass{amsart}
\usepackage{amsmath}
\usepackage{amssymb}
\usepackage{array}
\usepackage{amscd}
\usepackage{leftindex}
\usepackage{tikz-cd}
\usetikzlibrary{arrows.meta, graphs}
\usepackage{float}
\usepackage{calc}
\usetikzlibrary{shapes.geometric}

\DeclareMathAlphabet{\mathdutchcal}{U}{dutchcal}{m}{n}
\SetMathAlphabet{\mathdutchcal}{bold}{U}{dutchcal}{b}{n}
\DeclareMathAlphabet{\mathdutchbcal}{U}{dutchcal}{b}{n}

\numberwithin{equation}{section}

\theoremstyle{plain}
\newtheorem{theorem}{Theorem}
\newtheorem{lemma}{Lemma}
\newtheorem{proposition}{Proposition}
\newtheorem{corollary}{Corollary}

\theoremstyle{definition}
\newtheorem{definition}{Definition}

\theoremstyle{remark}
\newtheorem{remark}{Remark}
\newtheorem{example}{Example}

\newcommand{\comm}[1]{}

\usepackage{color}
\usepackage[normalem]{ulem}
\definecolor{DarkGreen}{rgb}{0,0.5,0.1} 

\newcommand\soutD{\bgroup\markoverwith
{\textcolor{DarkGreen}{\rule[.5ex]{2pt}{1pt}}}\ULon}

\makeatletter

\newcommand*{\rom}[1]{\expandafter\@slowromancap\romannumeral #1@}
\makeatother

\begin{document}
 
\title[The most symmetric cubic surface ]{The most symmetric smooth cubic surface}
\author{Anastasia V.~Vikulova}
\address{{\sloppy
\parbox{0.9\textwidth}{
Steklov Mathematical Institute of Russian
Academy of Sciences,
8 Gubkin str., Moscow, 119991, Russia
}\bigskip}}
\email{vikulovaav@gmail.com}
\date{}
\thanks{This work was supported by the Russian Science Foundation under grant no. 23-11-00033, https://rscf.ru/en/project/23-11-00033/}
\maketitle

\begin{abstract}
In this paper we give a classification of the largest automorphism groups of a smooth cubic surface over arbitrary fields.  Moreover, we prove that for a given field a smooth cubic surface with the largest automorphism group is unique up to isomorphism.
\end{abstract}

\section{Introduction}
\label{s1}

The protagonist of this paper is a~smooth cubic surface over an arbitrary field. Recall two famous 
cubic surfaces which have their own names. The cubic surface given by the equation
$$
x^3+y^3+z^3+t^3=0
$$
in $\mathbb{P}^3$ is called the \textit{Fermat cubic surface}, and the cubic surface given by the equations
$$
x+y+z+t+w=xyz+xyt+xyw+xzt+xzw+xtw+yzt+yzw+ytw+ztw=0
$$
in $\mathbb{P}^4$ is called the \textit{Clebsch cubic surface}\/.
The Fermat cubic surface is an example of a~cubic surface with the simplest equations on which all $27$ lines over $\mathbb{C}$ can easily be found. The Clebsch cubic contains $27$ real lines, and Klein presented a~model of it at the Chicago World's Fair in 1893. 
It is well known that the Fermat surface over an algebraically closed field of characteristic different from $3$ is the smooth cubic surface with the largest automorphism group (see, for instance,~\cite{7}, Theorem~9.4.1, and~\cite{8}, Theorem~1.1).
The automorphism group of the Clebsch cubic surface is isomorphic to $\mathfrak{S}_5$ over any field of characteristic different from $2$ and $5$ (see, for example,~\cite{8}, Theorem~6.1). Over a~field of characteristic different from $3$ the Clebsch cubic surface can be given by more simple equations
$$
x+y+z+t+w=x^3+y^3+z^3+t^3+w^3=0.
$$
Note that over a~field of characteristic~$5$ the Clebsch surface is singular at the point~\mbox{$[1:1:1:1:1] \in \mathbb{P}^4$.} The other two cubic surfaces that we need are given by
\begin{equation}
\label{eq1.1}
t^3+tz^2-xy^2+x^2z=0,
\end{equation}
and
\begin{equation}
\label{eq1.2}
 x^2t+y^2z+z^2y+t^2x=0.
\end{equation}

\goodbreak

While smooth cubic surfaces over algebraically closed fields have thoroughly been investigated from all points of view, the question of the automorphism groups of smooth cubic surfaces over arbitrary fields has virtually not been considered in the literature so far. The only exceptions are~\cite{16} and~\cite{18}.

To start with, we recall the results about the largest automorphism groups of smooth cubic surfaces over an algebraically closed field, which were established in~\cite{8} for a~positive characteristic and in~\cite{11} for characteristic zero.

\begin{theorem}[{see~\cite{8}, Theorem~1.1, and~\cite{11}, Theorem~5.3}]
\label{t1}
Let $S$ be a~smooth cubic surface over an algebraically closed field $\mathbf{F}$. Then the following hold.

$(\mathrm{i})'$ If\/ $\mathbf{F}$ is of characteristic~$2$, then $|{\operatorname{Aut}(S)}| \leqslant 25\, 920$. Equality holds if and only if~$\operatorname{Aut}(S)\simeq \mathrm{PSU}_4(\mathbb{F}_2)$ and $S$ is isomorphic to the Fermat cubic surface.

$(\mathrm{i})''$ If\/ $\mathbf{F}$ is of characteristic~$2$ and $|{\operatorname{Aut}(S)}| < 25\, 920$, then
$$
|{\operatorname{Aut}(S)}| \leqslant 192.
$$
Equality holds if and only if
$$
\operatorname{Aut}(S)\simeq (\mathbb{Z}/2\mathbb{Z})^3 \rtimes \mathfrak{S}_4.
$$

$(\mathrm{ii})'$ If\/ $\mathbf{F}$ is of characteristic~$3$, then $|{\operatorname{Aut}(S)}| \leqslant 216$. Equality holds if and only if~$\operatorname{Aut}(S) \simeq \mathcal{H}_3(\mathbb{F}_3) \rtimes \mathbb{Z}/8\mathbb{Z}$ and $S$ is isomorphic to a~cubic surface~\eqref{eq1.1}.

$(\mathrm{ii})''$ If\/ $\mathbf{F}$ is of characteristic~$3$ and $|{\operatorname{Aut}(S)}| < 216$, then $|{\operatorname{Aut}(S)}| \leqslant 120$. Equality holds if and only if~$\operatorname{Aut}(S) \simeq \mathfrak{S}_5$ and $S$ is isomorphic to the Clebsch cubic surface.

$(\mathrm{iii})'$ If\/ $\mathbf{F}$ is of characteristic~$5$, then $|{\operatorname{Aut}(S)}| \leqslant 648$. Equality holds if and only if $\operatorname{Aut}(S) \simeq (\mathbb{Z}/3\mathbb{Z})^3 \rtimes \mathfrak{S}_4$ and $S$ is isomorphic to the Fermat cubic surface.

$(\mathrm{iii})''$ If\/ $\mathbf{F}$ is of characteristic~$5$ and $|{\operatorname{Aut}(S)}| < 648$, then $|{\operatorname{Aut}(S)}| \leqslant 108$. Equality holds if and only if
$$
\operatorname{Aut}(S) \simeq \mathcal{H}_3(\mathbb{F}_3) \rtimes \mathbb{Z}/4\mathbb{Z}.
$$
Moreover, if $|{\operatorname{Aut}(S)}| < 108$, then $|{\operatorname{Aut}(S)}| \leqslant 54$. The equality holds if and only if
$$
\operatorname{Aut}(S) \simeq \mathcal{H}_3(\mathbb{F}_3) \rtimes \mathbb{Z}/2\mathbb{Z}.
$$

$(\mathrm{iv})'$ If\/ $\mathbf{F}$ is of characteristic different from $2$, $3$ and $5$, then $|{\operatorname{Aut}(S)}| \leqslant 648$. Equality holds if and only if
$$
\operatorname{Aut}(S) \simeq (\mathbb{Z}/3\mathbb{Z})^3 \rtimes \mathfrak{S}_4
$$
and $S$ is isomorphic to the Fermat cubic surface.

$(\mathrm{iv})''$ If\/ $\mathbf{F}$ is of characteristic different from $2$, $3$ and $5$ and $|{\operatorname{Aut}(S)}| < 648$, then~\mbox{$|{\operatorname{Aut}(S)}| \leqslant 120$.} Equality holds if and only if $\operatorname{Aut}(S) \simeq \mathfrak{S}_5$ and~$S$ is isomorphic to the Clebsch cubic surface.
\end{theorem}

Let us return to examples of cubic surfaces. The interesting feature of~\eqref{eq1.2} is that, up to isomorphism, it is the unique smooth cubic surface over~$\mathbb{F}_{2^{2k+1}}$ for any $k \in \mathbb{Z}_{\geqslant 0}$ with automorphism group of maximum order (see~\cite{18}, Theorem~1.4). This automorphism group is isomorphic to $\mathfrak{S}_6$. Moreover, by Corollary~1.7 in~\cite{18}, over~$\mathbb{F}_{4^{k}}$ for any $k \in \mathbb{Z}_{\geqslant 0}$ the surface~\eqref{eq1.2} is isomorphic to the Fermat cubic surface. As a~result, over a~finite field of characteristic~$2$ it is the smooth cubic surface with the largest automorphism group. To sum up, we have the following result.

\begin{theorem}[{\cite{18}, Theorem~1.4}]
\label{t2}
Let $S$ be a~smooth cubic surface over a~finite field~$\mathbf{F}$ of characteristic~$2$.

{\rm (i)}
If\/ $\mathbf{F}=\mathbb{F}_{4^k}$, then $|{\operatorname{Aut}(S)}| \leqslant 25\, 920$. Moreover,
$$
|{\operatorname{Aut}(S)}| = 25\, 920
$$
if and only if
$\operatorname{Aut}(S) \simeq \mathrm{PSU}_4(\mathbb{F}_2)$
and $S$ is isomorphic to the Fermat cubic surface.

{\rm (ii)}
If\/ $\mathbf{F}=\mathbb{F}_{2^{2k+1}}$, then $|{\operatorname{Aut}(S)}| \leqslant 720$. Moreover,
$$
|{\operatorname{Aut}(S)}| = 720
$$
if and only if $\operatorname{Aut}(S) \simeq \mathfrak{S}_6$ and $S$ is isomorphic to cubic surface~\eqref{eq1.2}.
\end{theorem}

The goal of this paper is to find the largest automorphism group of a~smooth cubic surface over an arbitrary field. So we are going to prove the following theorem.

\begin{theorem}
\label{t3}
Let $S$ be a~smooth cubic surface over a~field $\mathbf{F}$.

Assume that the characteristic of $\mathbf{F}$ is equal to $2$.

$(\mathrm{i})_2$ If\/ $\mathbf{F}$ contains a~nontrivial cube root of unity, then $|{\operatorname{Aut}(S)}| \leqslant 25\, 920$. Moreover, $|{\operatorname{Aut}(S)}| = 25\, 920$ if and only if
$$
\operatorname{Aut}(S) \simeq \mathrm{PSU}_4(\mathbb{F}_2)
$$
and $S$ is isomorphic to the Fermat cubic surface.

$(\mathrm{ii})_2$
If\/ $\mathbf{F}$ does not contain a~nontrivial cube root of unity, then
$|{\operatorname{Aut}(S)}| \leqslant 720$. Moreover, $|{\operatorname{Aut}(S)}| = 720$ if and only if
$$
\operatorname{Aut}(S) \simeq \mathfrak{S}_6
$$
and $S$ is isomorphic to the surface~\eqref{eq1.2}.

Assume that the characteristic of $\mathbf{F}$ is equal to $3$.

$(\mathrm{i})_3$  If $\mathbf{F}$ contains a~nontrivial fourth root of unity, then $|{\operatorname{Aut}(S)}| \leqslant 216$. Moreover,~\mbox{$|{\operatorname{Aut}(S)}| = 216$} if and only if
$$
\operatorname{Aut}(S) \simeq \mathcal{H}_3(\mathbb{F}_3) \rtimes \mathbb{Z}/8\mathbb{Z}
$$
and $S$ is isomorphic to the smooth surface~\eqref{eq1.1}.

$(\mathrm{ii})_3$ If\/ $\mathbf{F}$ contains no nontrivial fourth roots of unity, then $|{\operatorname{Aut}(S)}| \leqslant 120$. Moreover, $|{\operatorname{Aut}(S)}| = 120$ if and only if
$$
\operatorname{Aut}(S) \simeq \mathfrak{S}_5
$$
and $S$ is isomorphic to the Clebsch cubic surface.

Assume that the characteristic of $\mathbf{F}$ is equal to $5$.

$(\mathrm{i})_5$ If\/ $\mathbf{F}$ contains a~nontrivial cube root of unity, then $|{\operatorname{Aut}(S)}| \leqslant 648$. Moreover,~\mbox{$|{\operatorname{Aut}(S)}| = 648$} if and only if
$$
\operatorname{Aut}(S) \simeq (\mathbb{Z}/3\mathbb{Z})^3 \rtimes \mathfrak{S}_4
$$
and $S$ is isomorphic to the Fermat cubic surface.

$(\mathrm{ii})_5$  If $\mathbf{F}$ contains no nontrivial cube roots of unity, then $|{\operatorname{Aut}(S)}| \leqslant 72$. Moreover,~\mbox{$|{\operatorname{Aut}(S)}| = 72$} if and only if
$$
\operatorname{Aut}(S) \simeq (\mathbb{Z}/3\mathbb{Z})^2 \rtimes \mathfrak{D}_4
$$
and $S$ is isomorphic to the surface~\eqref{eq1.2}.

\goodbreak

Assume that the characteristic of $\mathbf{F}$ is different from $2$, $3$ and~$5$.

$(\mathrm{i})_{\neq 2,3,5}$ If\/ $\mathbf{F}$ contains a~nontrivial cube root of unity, then $|{\operatorname{Aut}(S)}| \leqslant 648$. Moreover,~$|{\operatorname{Aut}(S)}| = 648$ if and only if
$$
\operatorname{Aut}(S) \simeq (\mathbb{Z}/3\mathbb{Z})^3 \rtimes \mathfrak{S}_4
$$
and $S$ is isomorphic to the Fermat cubic surface.

$(\mathrm{ii})_{\neq 2,3,5}$ If\/ $\mathbf{F}$ contains no nontrivial cube roots of unity, then $|{\operatorname{Aut}(S)}| \leqslant 120$. Moreover, $|{\operatorname{Aut}(S)}| = 120$ if and only if
$$
\operatorname{Aut}(S) \simeq \mathfrak{S}_5
$$
and $S$ is isomorphic to the Clebsch cubic surface.
\end{theorem}

Note that in~\cite{16} a~classification of automorphism groups of smooth cubic surfaces over all fields of characteristic zero was presented. One can deduce Theorem~\ref{t3} for fields of characteristic zero from these results (see~\cite{16}, Theorem~1.1).

Theorem~\ref{t3} generalizes Theorem~\ref{t2}. In particular, from Theorem~\ref{t3} we immediately deduce Corollary~\ref{c1}, which, in combination with Theorem~\ref{t2}, completes the classification of the largest automorphism groups of smooth cubic surfaces over finite fields.

\begin{corollary}
\label{c1}
Let $S$ be a~smooth cubic surface over a~finite field $\mathbf{F}$ of odd characteristic.

Assume that the characteristic of $\mathbf{F}$ is equal to $3$.

$(\mathrm{i})_3$ If\/ $\mathbf{F}=\mathbb{F}_{9^k}$, then $|{\operatorname{Aut}(S)}| \leqslant 216$. Moreover,$|{\operatorname{Aut}(S)}| = 216$ if and only if
$$
\operatorname{Aut}(S) \simeq \mathcal{H}_3(\mathbb{F}_3) \rtimes \mathbb{Z}/8\mathbb{Z}
$$
 and $S$ is isomorphic to the smooth cubic surface~\eqref{eq1.1}.

$(\mathrm{ii})_3$
If\/ $\mathbf{F}=\mathbb{F}_{3^{2k+1}}$, then $|{\operatorname{Aut}(S)}| \leqslant 120$. Moreover, $|{\operatorname{Aut}(S)}| = 120$ if and only if 
$$
\operatorname{Aut}(S) \simeq \mathfrak{S}_5
$$
 and $S$ is isomorphic to the Clebsch cubic surface.

Assume that the characteristic of $\mathbf{F}$ is equal to~$5$.

$(\mathrm{i})_5$ If\/ $\mathbf{F}=\mathbb{F}_{25^k}$, then $|{\operatorname{Aut}(S)}| \leqslant 648$. Moreover, $|{\operatorname{Aut}(S)}| = 648$ if and only if 
$$
\operatorname{Aut}(S) \simeq (\mathbb{Z}/3\mathbb{Z})^3 \rtimes \mathfrak{S}_4
$$
 and $S$ is isomorphic to the Fermat cubic surface.

$(\mathrm{ii})_5$ If\/ $\mathbf{F}=\mathbb{F}_{5^{2k+1}}$, then $|{\operatorname{Aut}(S)}| \leqslant 72$. Moreover, $|{\operatorname{Aut}(S)}| = 72$ if and only if 
$$
\operatorname{Aut}(S) \simeq (\mathbb{Z}/3\mathbb{Z})^2 \rtimes \mathfrak{D}_4
$$
 and $S$ is isomorphic to the surface~\eqref{eq1.2}.

\goodbreak

Assume that the characteristic of $\mathbf{F}$ is at least~$7$.

$(\mathrm{i})_{\geqslant 7}$ If the cardinality of $\mathbf{F}$ is congruent to $1$ modulo $3$, then $|{\operatorname{Aut}(S)}| \leqslant 648$. Moreover, $|{\operatorname{Aut}(S)}| = 648$ if and only if
$$
\operatorname{Aut}(S) \simeq (\mathbb{Z}/3\mathbb{Z})^3 \rtimes \mathfrak{S}_4
$$
and $S$ is isomorphic to the Fermat cubic surface.

\goodbreak

$(\mathrm{ii})_{\geqslant 7}$ If the cardinality of $\mathbf{F}$ is congruent to~$2$ modulo~$3$, then $|{\operatorname{Aut}(S)}| \leqslant 120$. Moreover, $|{\operatorname{Aut}(S)}| = 120$ if and only if
$$
\operatorname{Aut}(S) \simeq \mathfrak{S}_5
$$
and $S$ is isomorphic to the Clebsch cubic surface.
\end{corollary}

In the following proposition we study isomorphism relations among the Fermat cubic surface, the Clebsch cubic surface and the surfaces~\eqref{eq1.1} and~\eqref{eq1.2} over fields of arbitrary characteristic.

\begin{proposition}
\label{p1}
Let $\mathbf{F}$ be a~field. Then the following hold.

{\rm (i)} The Fermat cubic surface is singular (in fact, nonreduced) if and only if\/ $\mathbf{F}$~is of characteristic~$3$. The Clebsch cubic surface is singular if and only if\/ $\mathbf{F}$ has characteristic~$5$. The cubic surface~\eqref{eq1.1} is singular if and only if\/ $\mathbf{F}$ has characteristic~$2$. The cubic surface~\eqref{eq1.2} is singular if and only if\/ $\mathbf{F}$ has characteristic~$3$.

{\rm (ii)} The surface~\eqref{eq1.1} is isomorphic neither to the Fermat cubic surface, nor to the Clebsch cubic surface, nor to the surface~\eqref{eq1.2}.

{\rm (iii)} If\/ $\mathbf{F}$ is of characteristic~$2$, then the surface~\eqref{eq1.2} and the Clebsch cubic surface are isomorphic.

{\rm (iv)} If\/ $\mathbf{F}$ is of characteristic different from~$3$, then the surface~\eqref{eq1.2} is isomorphic to the Fermat cubic surface if and only if\/ $\mathbf{F}$ contains a~nontrivial cube root of unity.

{\rm (v)} If\/ $\mathbf{F}$ is of characteristic different from~$2$, then the Clebsch cubic surface is isomorphic neither to the Fermat cubic surface, nor to the surface~\eqref{eq1.2}.
\end{proposition}

For the reader's convenience we rephrase the contents of Theorem~\ref{t3}, together with the additional information from Proposition~\ref{p1}, in Table~\ref{tab1}.

\begin{table}[H]
\begin{center}
\setlength\extrarowheight{5pt}
\begin{tabular}{|c|c|c|c|}
\hline
 & \text{Largest} & \text{Order} & \text{Cubic surface with} \\
\text{Field} & \text{automorphism} & \text{of $G$} & \text{ automorphism group $G$} \\
 & \text{group $G$} & &  \\
\hline
$char(\mathbf{F})=2, \, \omega \in \mathbf{F}$ & $\mathrm{PSU}_4(\mathbb{F}_2)$ & $25\,920$ & \text{Fermat cubic surface} $=$~\eqref{eq1.2}$=$ \\[5pt]
& & & $=$ \text{Clebsch cubic surface} \\[5pt]
\hline
$char(\mathbf{F})=2, \, \omega \notin \mathbf{F}$ & $\mathfrak{S}_6$ & $720$ & \text{Clebsch cubic surface} $=$~\eqref{eq1.2} \\[5pt]
\hline
$char(\mathbf{F})=3, \, \mathbf{i} \in \mathbf{F} $ & $\mathcal{H}_3(\mathbb{F}_3) \rtimes \mathbb{Z}/8\mathbb{Z}$ & $216$ & \eqref{eq1.1}\\[5pt]
\hline
$char(\mathbf{F})=3, \, \mathbf{i} \notin \mathbf{F} $ & $\mathfrak{S}_5$ & $120$ & \text{Clebsch cubic surface} \\[5pt]
\hline
$char(\mathbf{F})=5, \, \omega \in \mathbf{F} $ & $ (\mathbb{Z}/3\mathbb{Z})^3 \rtimes \mathfrak{S}_4$ & $648$ & \text{Fermat cubic surface} $=$~\eqref{eq1.2} \\[5pt]
\hline
$char(\mathbf{F})=5, \, \omega \notin \mathbf{F} $ & $ (\mathbb{Z}/3\mathbb{Z})^2 \rtimes \mathfrak{D}_4$ & $72$ & \eqref{eq1.2}\\[5pt]
\hline
$char(\mathbf{F}) \neq 2,3,5, \, \omega \in \mathbf{F} $ & $ (\mathbb{Z}/3\mathbb{Z})^3 \rtimes \mathfrak{S}_4$ & $648$ & \text{Fermat cubic surface} $=$~\eqref{eq1.2} \\[5pt]
\hline
$char(\mathbf{F}) \neq 2,3,5, \, \omega \notin \mathbf{F} $ & $\mathfrak{S}_5$ & $120$ & \text{Clebsch cubic surface}  \\[5pt]

\hline
\end{tabular}
\vspace*{3mm}
\caption{Table of the largest automorphism groups.}
\label{tab1}
\end{center}
\end{table}

Here $\omega$ is a~nontrivial cube root of unity and $\mathbf{i}$ is a~primitive fourth root of unity.

We describe the structure of the paper. In \S\,\ref{s2} we prove some
auxiliary lemmas and recall some facts about the Weyl group
$W(\mathrm{E}_6)$, its subgroups and elements. In~\S\,\ref{s3} we recall some
basic facts about smooth cubic surfaces and discuss their automorphisms. In
\S\,\ref{s4} we study smooth cubic surfaces with a~regular action
of~$\mathfrak{S}_5$ and prove that such cubic surfaces are isomorphic to the
Clebsch cubic surface. In \S\,\ref{s5} we study smooth cubic surfaces and
their automorphisms over fields of characteristic~$3$. In \S\,\ref{s6} we
study smooth cubic surfaces with a~regular action
of~$(\mathbb{Z}/3\mathbb{Z})^3 \rtimes \mathfrak{S}_4$ and prove that such
cubic surfaces are isomorphic to the Fermat cubic surface. In~\S\,\ref{s7} we
study smooth cubic surfaces and their automorphism groups over fields that do
not contain a~nontrivial cube root of unity. In \S\,\ref{s8} we prove
Theorem~\ref{t3} and Proposition~\ref{p1}.

\subsection*{Notation}
Let $X$ be a~variety defined over a~field $\mathbf{F}$. If $\mathbf{F} \subset \mathbf{L}$ is an extension of~$\mathbf{F}$, then we denote by $X_{\mathbf{L}}$ the variety
$$
X_{\mathbf{L}}=X \times_{\mathrm{Spec}(\mathbf{F})}\mathrm{Spec}(\mathbf{L})
$$
over $\mathbf{L}$. We denote the $\mathbf{F}$-rational points
on~$X$ by $X(\mathbf{F})$, the algebraic closure of~$\mathbf{F}$ by $\overline{\mathbf{F}}$ and the separable closure
of $\mathbf{F}$ by $\mathbf{F}^{\mathrm{sep}}$. We let $\mathfrak{S}_n$ denote the symmetric group of
degree~$n$, $\mathfrak{A}_n$ denote the alternating group of degree~$n$
and $\mathfrak{D}_n$ denote the dihedral group of order~$2n$. We denote the Heisenberg group of upper
triangular~$3\times 3$ matrices over the field~$\mathbb{F}_3$ with diagonal
entries equal to~$1$ by
$\mathcal{H}_3(\mathbb{F}_3)$.

We denote a~nontrivial cube root of unity by $\omega$,
 a~primitive fourth root of unity by $\mathbf{i}$, and a~primitive $n$th root of unity by $\zeta_n$.

\subsection*{Acknowledgment}
The author would like to extend her sincere gratitude to C.\,A.~Shramov for proposing the problem, his constant support and for teaching the author how to write readable mathematical papers. Also the author would like to thank A.\,S.~Trepalin for the deep reading of the paper and useful advice. Finally the~author would like to extend her deepest gratitude to the referees for reading~the manuscript carefully and giving extremely important and detailed comments, which helped her to generalize the result of the rough draft.

\section{Preliminaries}
\label{s2}

In this section we prove some auxiliary lemmas and recall some facts concerning the Weyl group $W(\mathrm{E}_6)$. Let us start with a~general observation on del~Pezzo surfaces of small degree; almost everywhere below we use it without explicit reference.

\begin{theorem}[{cf.~\cite{13}, Theorem 4.2, and~\cite{15}, Lemma~7.1}]
\label{t4}
Let $X$ be a~del~Pezzo surface of degree~$d \leqslant 5$ over a~field~$\mathbf{F}$. Then the following holds.
\begin{enumerate}
\item[{\rm (i)}] $\operatorname{Aut}(X_{\overline{\mathbf{F}}}) \simeq \operatorname{Aut}(X_{\mathbf{F}^{\mathrm{sep}}})$.

\item[{\rm (ii)}] The lines on $X_{\overline{\mathbf{F}}}$ are defined over $\mathbf{F}^{\mathrm{sep}}$.

\item[{\rm (iii)}] $\mathrm{Pic}(X_{\overline{\mathbf{F}}}) \simeq \mathrm{Pic}(X_{\mathbf{F}^{\mathrm{sep}}})$.
\end{enumerate}
\end{theorem}

\begin{proof}
The first two assertions can be proved as follow. We denote by $\mathcal{S}$ either~$\operatorname{Aut}(X_{\overline{\mathbf{F}}})$, or the Hilbert scheme~$\mathcal{H}ilb_{l,X_{\overline{\mathbf{F}}}}$ of lines on $X_{\overline{\mathbf{F}}}$. The scheme $\mathcal{S}$ is zero-\allowbreak dimensional. First of all, we show that $\mathcal{S}$ is a~smooth scheme. Then applying Corollary~2.2.13 from~\cite{3} we see that the set of points defined over $\mathbf{F}^{\mathrm{sep}}$ on~$\mathcal{S}$ is dense. Since $\mathcal{S}$ is finite, all points are defined over $\mathbf{F}^{\mathrm{sep}}$.

Let us prove the smoothness of $\mathcal{S}$. It is sufficient to prove that the Zariski tangent space $T_p\mathcal{S}$ to $\mathcal{S}$ at any point $p \in \mathcal{S}$ satisfies the equality
\begin{equation}
\label{eq2.1}
\dim T_p \mathcal{S}=0.
\end{equation}

Assume that $\mathcal{S}=\operatorname{Aut}(X_{\overline{\mathbf{F}}})$. It is well known that
$$
T_p \operatorname{Aut}(X_{\overline{\mathbf{F}}}) \simeq H^0(X_{\overline{\mathbf{F}}}, TX_{\overline{\mathbf{F}}}),
$$
where $TX_{\overline{\mathbf{F}}}$ is the tangent bundle of $X_{\overline{\mathbf{F}}}$. Since $X_{\overline{\mathbf{F}}}$ is a~del~Pezzo surface of degree~$d$ not exceeding $5$ over an algebraically closed field, it is a~blowup of $9-d>3$ points. The second Chern class of $T{\mathbb{P}^2}$ is equal to~$3$. Therefore, there are~no~global~sections~of the~tangent bundle of $\mathbb{P}^2$ vanishing at $9-d$ points. Since~the~global~sections of~$TX_{\overline{\mathbf{F}}}$ correspond to global sections on $T\mathbb{P}^2$ vanishing at $9-d$ points, we get that $ H^0(X_{\overline{\mathbf{F}}}, TX_{\overline{\mathbf{F}}})=0$. Therefore, we obtain equality~\eqref{eq2.1} and, as a~result, the smoothness of the scheme $\operatorname{Aut}(X_{\overline{\mathbf{F}}})$.

Assume that $\mathcal{S}=\mathcal{H}ilb_{l,X_{\overline{\mathbf{F}}}}$. It is well known that
$$
T_l \mathcal{H}ilb_{l,X_{\overline{\mathbf{F}}}} \simeq H^0(l, \mathcal{N}_{l/X_{\overline{\mathbf{F}}}}),
$$
where $\mathcal{N}_{l/X_{\overline{\mathbf{F}}}}$ is the normal bundle to the line $l$ in $X_{\overline{\mathbf{F}}}$. The degree of the bundle~$\mathcal{N}_{l/X_{\overline{\mathbf{F}}}}$ is equal to~$-1$. Therefore, $H^0(X_{\overline{\mathbf{F}}}, TX_{\overline{\mathbf{F}}})=0$. So we obtain~\eqref{eq2.1} and, as a~result, the smoothness of the scheme $\mathcal{H}ilb_{l,X_{\overline{\mathbf{F}}}}$.

Assertion~(iii) follows from~(ii).

The theorem is proved.
\end{proof}

\begin{remark}
\label{r1}
From Theorem~\ref{t4},\,(i), we immediately get that Theorem~\ref{t1} holds under the assumption that the field $\mathbf{F}$ is separably closed instead of algebraically closed.
\end{remark}

Below we study conjugate elements in
$\mathrm{PGL}_n(\overline{\mathbf{F}})$ over~$\mathbf{F}$
and~$\mathbf{F}^{\mathrm{sep}}$. First of all, we need the following
auxiliary lemma about conjugate matrices over $\mathbf{F}$ and over an
extension $\mathbf{F} \subset \mathbf{L}$.

\begin{lemma}
\label{l1}
Let $\mathbf{F} \subset \mathbf{L}$ be a~field extension. Let us consider two matrices $A$ and~$B$  in~$\mathrm{Mat}_{n \times n}(\mathbf{F})$. Assume that they are conjugate in $\mathrm{Mat}_{n \times n}(\mathbf{L})$. Then they are conjugate in $\mathrm{Mat}_{n \times n}(\mathbf{F})$.
\end{lemma}

\begin{proof}
Consider the matrices $A$ and $B$ as linear operators on vector spaces~$\mathbf{F}^n$ and~$\mathbf{L}^n$. As $A$ and $B$ are conjugate over $\mathbf{L}$, by Theorem~15.1 in~\cite{10} these operators are conjugate to multiplication by $t$ in the direct sum of residue modules
\begin{equation}
\label{eq2.2}
\frac{\mathbf{L}[t]}{p_1^{m_1}(t)} \oplus \dots \oplus \frac{\mathbf{L}[t]}{p_k^{m_k}(t)},
\end{equation}
where the $p_i(t)$ for $i=1, \dots, k$ are monic irreducible polynomials over
$\mathbf{L}$. Note that it is not required that all the $p_i(t)$ be different. By Corollary 15.4 in~\cite{10} the characteristic polynomial of both $A$ and $B$ is equal to
$$
p_1^{m_1}(t) \dotsb p_k^{m_k}(t).
$$

By Theorem 15.1 in~\cite{10} the operator $A$ over $\mathbf{F}$ is conjugate to multiplication by~$t$ in the direct sum of residue modules
\begin{equation}
\label{eq2.3}
\frac{\mathbf{F}[t]}{f_1^{u_1}(t)} \oplus \dots \oplus \frac{\mathbf{F}[t]}{f_s^{u_s}(t)}
\end{equation}
and the operator $B$ over $\mathbf{F}$ is conjugate to multiplication by $t$ in the direct sum of residue modules
\begin{equation}
\label{eq2.4}
\frac{\mathbf{F}[t]}{g_1^{v_1}(t)} \oplus \dots \oplus \frac{\mathbf{F}[t]}{g_r^{v_r}(t)},
\end{equation}
where $f_i(t)$ and $g_j(t)$ for $i=1, \dots, s$ and $j=1, \dots, r$ are monic irreducible polynomials over $\mathbf{F}$. Since
$$
p_1^{m_1}(t) \dotsb p_k^{m_k}(t)=f_1^{u_1}(t) \dotsb f_s^{u_s}(t) = g_1^{v_1}(t) \dotsb g_r^{v_r}(t),
$$
we get that the sets of different irreducible polynomials among $f_1(t), \dots, f_s(t)$ and among $g_1(t), \dots, g_r(t)$ coincide.
After tensoring both~\eqref{eq2.3} and~\eqref{eq2.4} by $\mathbf{L}$ we obtain~the direct sum~\eqref{eq2.2}. This means that for any irreducible polynomial over~$\mathbf{F}$ the~sets of powers
in which it appears in~\eqref{eq2.3} and~\eqref{eq2.4} coincide. This means that the sums~\eqref{eq2.3} and~\eqref{eq2.4} are isomorphic, so the matrices $A$ and $B$ are conjugate in~$\mathrm{Mat}_{n \times n}(\mathbf{F})$.

The lemma is proved.
\end{proof}

\begin{remark}
\label{r2}
Recall what conjugacy in the projective linear group means. Let $\mathbf{F}$ be a~field. Consider the natural homomorphism
$$
\mathrm{GL}_n(\mathbf{F}) \stackrel{\rho}{\twoheadrightarrow} \mathrm{PGL}_n(\mathbf{F}).
$$
Let $a$ and $b$ be two elements of $\mathrm{PGL}_n(\mathbf{F})$, and let $A \in \rho^{-1}(a)$ and $B \in \rho^{-1}(b)$. Two elements $a,b \in \mathrm{PGL}_n(\mathbf{F})$ are conjugate if there are $\lambda \in \mathbf{F}^*$ and $C \in \mathrm{GL}_n(\mathbf{F})$ such that
$$
 CAC^{-1}=\lambda B.
$$
\end{remark}

\begin{lemma}
\label{l2}
Let $\mathbf{F}$ be a~field of characteristic~$p \neq 3$. Let $\mathbf{F} \subset \mathbf{L}$ be a~field extension. Let $a$ and $b$ be two elements of order~$3$ in $\mathrm{PGL}_4(\mathbf{F})$ such that they are conjugate in~$\mathrm{PGL}_4(\mathbf{L})$. Then they are conjugate in $\mathrm{PGL}_4(\mathbf{F})$.
\end{lemma}

\begin{proof}
Consider the natural homomorphism
$$
\mathrm{GL}_4(\mathbf{F}) \stackrel{\rho}{\twoheadrightarrow} \mathrm{PGL}_4(\mathbf{F}).
$$
Let $A \in \rho^{-1}(a)$ and $B \in \rho^{-1}(b)$. As $a, b \in \mathrm{PGL}_4(\mathbf{F})$ are conjugate
in $\mathrm{PGL}_4(\mathbf{L})$, there are $\lambda \in \mathbf{L}^*$ and~$C \in \mathrm{GL}_4(\mathbf{L})$ such that
\begin{equation}
\label{eq2.5}
CAC^{-1}=\lambda B.
\end{equation}

First of all, let us prove that $\lambda \in \mathbf{F}^*$. As $a$ and $b$ are of order~$3$, we have $A^3=sE$ and $B^3=tE$, where $s,t \in \mathbf{F}^*$ and $E$ is the identity matrix. Therefore, cubing~\eqref{eq2.5} we obtain
$$
sE=A^3=\lambda^3 B^3=\lambda^3 tE.
$$

This means that
\begin{equation}
\label{eq2.6}
\lambda^3=\frac{s}{t} \in \mathbf{F}^*.
\end{equation}
Taking the determinant of~\eqref{eq2.5} we obtain the equality
$$
 \lambda^4 \det(B)=\det(A).
$$
This means that
\begin{equation}
\label{eq2.7}
\lambda^4=\frac{\det(A)}{\det(B)} \in \mathbf{F}^*.
\end{equation}

Dividing~\eqref{eq2.7} by~\eqref{eq2.6} we see that
$$
\lambda=\frac{t\det(A)}{s\det(B)} \in \mathbf{F}^*.
$$

Fix some $\lambda$ satisfying~\eqref{eq2.5}. Applying Lemma~\ref{l1} to the matrices $A$ and~$\lambda B$ we get that there
exists $C \in \mathrm{GL}_4(\mathbf{F})$ such that~\eqref{eq2.5} holds. This means that the elements~$a$ and~$b$ are conjugate in $\mathrm{PGL}_4(\mathbf{F})$.

Lemma~\ref{l2} is proved.
\end{proof}

Recall the results from~\cite{9}, Table 9, or~\cite{17}, Theorem~1, about the conjugacy classes of elements of
 the Weyl group~$W(\mathrm{E}_6)$.

\begin{lemma}
\label{l3}
The conjugacy classes of elements of the Weyl group~$W(\mathrm{E}_6)$ of orders~$2$,~$3$,~$5$,~$8$,~$9$ and~$10$, their
cardinalities, the orders of their centralizers, the sets of eigenvalues of the representation on the root
lattice $\mathrm{E}_6$ are as presented in Table~\ref{tab2}.
\end{lemma}

\begin{table}[H]
\begin{center}
\setlength\extrarowheight{5pt}
\begin{tabular}{|c|c|c|c|c|}
\hline
\text{Order of} & \text{Conjugacy class} & \text{Cardinality} & \text{Order of} & \text{Eigenvalues}\\
\text{element} &  &  &  \text{centralizer} & \\
 \hline
$2$ & $A_1$ & $36$  & $1440$ & $-1,1,1,1,1,1$  \\[5pt]
\hline
$2$ & $A_1^2$ & $270$ & $192$ & $-1,-1,1,1,1,1$ \\[5pt]
\hline
$2$ & $A_1^3$ & $540$ & $96$ & $-1,-1,-1,1,1,1$  \\[5pt]
 \hline
$2$ & $A_1^4$ & $45$ & $1152$ & $-1,-1,-1,-1,1,1$  \\[5pt]
\hline
$3$ & $A_2$ & $240$  & $216$ & $\omega,\omega^2,1,1,1,1$  \\[5pt]
\hline
$3$ & $A_2^2$ & $480$ & $108$ & $\omega,\omega,\omega^2,\omega^2,1,1$ \\[5pt]
\hline
$3$ & $A_2^3$ & $80$ & $648$ & $\omega,\omega,\omega,\omega^2,\omega^2,\omega^2$  \\[5pt]
\hline
$5$ & $A_4$ & $5184$ & $10$ & $\zeta_5,\zeta_5^2,\zeta_5^3,\zeta_5^4,1,1$  \\[5pt]
\hline
$8$ & $D_5$ & $6480$ & $8$ & $\zeta_8, \zeta_8^3, \zeta_8^5, \zeta_8^7, -1,1$  \\[5pt]
\hline
$9$ & $E_6(a_1)$ & $5760$ & $9$ & $\zeta_9, \zeta_9^2, \zeta_9^4, \zeta_9^5, \zeta_9^7, \zeta_9^8$  \\[5pt]
\hline
$10$ & $A_4 \times A_1$ & $5184$ & $10$ & $\zeta_5,\zeta_5^2,\zeta_5^3,\zeta_5^4,-1,1$  \\[5pt]
\hline
\end{tabular}
\vspace*{3mm}
\caption{Table of conjugacy classes of elements of order $2$, $3$, $5$, $8$, $9$ and~$10$ in $W(\mathrm{E}_6)$.}
\label{tab2}
\end{center}
\end{table}

There are seven conjugacy classes of elements of order~$6$ in $W(\mathrm{E}_6)$, however, we need only one
conjugacy class of elements of order~$6$, which we denote by $E_6(a_2)$. 

\begin{lemma}[{\cite{17}, Theorem~1}]
\label{l4}
Let $g$ be an element of $W(\mathrm{E}_6)$ of order~$6$ lying in the conjugacy class $E_6(a_2)$. Then its eigenvalues of the representation on the root lattice $\mathrm{E}_6$ are
$$
-\omega,\ -\omega,\ -\omega^2,\ -\omega^2,\ \omega,\ \omega^2,
$$
where $\omega \in \mathbb{C}$ is a~nontrivial cube root of unity.
\end{lemma}

\begin{lemma}
\label{l5}
Let $g \in W(\mathrm{E}_6)$.
\begin{enumerate}
\item[{\rm (i)}] If $g$ lies in the conjugacy class $E_6(a_2)$, then $g^2$ lies in the conjugacy class~$A_2^3$ and $g^3$ lies in the conjugacy class $A_1^4$.

\item[{\rm (ii)}] If $g$ is of order~$8$, then $g^4$ lies in the conjugacy class $A_1^4$.

\item[{\rm (iii)}] If $g$ is of order~$10$, then $g^5$ lies in the conjugacy class $A_1$.
\end{enumerate}
\end{lemma}

\begin{proof}
This follows directly from Lemmas~\ref{l3} and~\ref{l4}. For the reader's convenience let us prove~(i). By
Lemma~\ref{l4} the set of eigenvalues of $g$ on the root lattice~$\mathrm{E}_6$ is
$$
-\omega,\ -\omega,\ -\omega^2,\ -\omega^2,\ \omega,\ \omega^2.
$$
Therefore, the set of eigenvalues of $g^2$ is
$$
\omega,\ \omega,\ \omega,\ \omega^2,\ \omega^2,\ \omega^2.
$$
By Lemma~\ref{l3} this is the set of eigenvalues of an element of the conjugacy class~$A_2^3$.
Similarly, the set of eigenvalues of $g^3$ is
$$
-1,\ -1,\ -1,\ -1,\ 1,\ 1,
$$
which is the set of eigenvalues of an element of the conjugacy class~$A_1^4$.

The proof of assertions~(ii) and~(iii) is quite analogous.

Lemma~\ref{l5} is proved.
\end{proof}

\begin{lemma}
\label{l6}
 The centralizer of an element of order~$5$ in $W(\mathrm{E}_6)$ is isomorphic to~$\mathbb{Z}/10\mathbb{Z}$.
\end{lemma}

\begin{proof}
By Lemma~\ref{l3} the centralizer of an element $g$ of order~$5$ in~$W(\mathrm{E}_6)$ has order~$10$. This means that the centralizer of $g$ in the Weyl group is either isomorphic to $\mathbb{Z}/10\mathbb{Z}$, or to the dihedral group of order~$10$, since these are
the only two groups of order~$10$ up to isomorphism. However, the second case is impossible because $g$ lies in its centralizer.

The lemma is proved.
\end{proof}

\begin{lemma}
\label{l7}
The centralizer of any subgroup in $W(\mathrm{E}_6)$ isomorphic to $\mathfrak{S}_5$ lies in a~subgroup isomorphic to $\mathbb{Z}/2\mathbb{Z}$.
\end{lemma}

\begin{proof}
The centralizer $\mathrm{C}_{W(\mathrm{E}_6)}(\mathfrak{S}_5)$ of the subgroup $\mathfrak{S}_5$ in $W(\mathrm{E}_6)$ lies in the centralizer $\mathrm{C}_{W(\mathrm{E}_6)}(\tau)$ of an element $\tau$ of order~$5$ of $\mathfrak{S}_5$. By Lemma~\ref{l6} we have the isomorphism $\mathrm{C}_{W(\mathrm{E}_6)}(\tau) \simeq \mathbb{Z}/10\mathbb{Z}$. Since
$\tau \in \mathrm{C}_{W(\mathrm{E}_6)}(\tau)$ and $\tau$ does not lie in the~centre of
$\mathfrak{S}_5$, we obtain the required result.

The lemma is proved.
\end{proof}

\goodbreak

From~\cite{17}, Theorem~1, we immediately deduce the following useful lemma.

\begin{lemma}
\label{l8}
There are no elements of order~$24$ in the Weyl group $W(\mathrm{E}_6)$.
\end{lemma}

Now we recall the classification of maximal proper subgroups of $\mathrm{PSU}_4(\mathbb{F}_2)$.

\begin{lemma}[{see~\cite{6}, p.\,26}]
\label{l9}
The maximal proper subgroups of $\mathrm{PSU}_4(\mathbb{F}_2)$ are as given in Table~\ref{tab3}.

\vskip0.2cm

\begin{table}[!h]
\vskip-4mm
\renewcommand{\arraystretch}{1.2}
\centering\begin{tabular}[c]{| >{\small}c| >{\small}c|}
\hline
\text{Subgroup} & \text{Order}\\
\hline
$(\mathbb{Z}/2\mathbb{Z})^4 \rtimes \mathfrak{A}_5$ & $960$ \\
\hline
$\mathfrak{S}_6$ & $720$ \\
\hline
$(\mathcal{H}_3(\mathbb{F}_3) \rtimes \mathbb{Z}/2\mathbb{Z})_{\bullet}\mathfrak{A}_4$ & $648$ \\
\hline
$(\mathbb{Z}/3\mathbb{Z})^3 \rtimes \mathfrak{S}_4$ & $648$ \\
\hline
$( \mathbb{Z}/2\mathbb{Z}^{\bullet}(\mathfrak{A}_4 \times \mathfrak{A}_4))_{\bullet}\mathbb{Z}/2\mathbb{Z}$ & $576$ \\
\hline
\end{tabular}\par
\vskip3mm
\caption{The maximal subgroups of $\mathrm{PSU}_4(\mathbb{F}_2)$}
\label{tab3}
\end{table}

Here we denote by $G_{\bullet} H$ a~group with normal subgroup isomorphic to $G$ and quotient
isomorphic to $H$, and we denote by $G^{\bullet} H$ the group $G_{\bullet} H$ that is not a~split
extension. All subgroups in Table~\ref{tab3} are unique up to conjugation.
\end{lemma}

Recall a~well-known result due to Beauville.

\begin{proposition}[{\cite{2}, Proposition~1.1}]
\label{p2}
Let $\mathbf{F}$ be a~field of characteristic $p \geqslant 0$, and let $n \in \mathbb{N}$. If $p>0$,
 assume that $\gcd(n,p)=1$. Then the group $\mathrm{PGL}_2(\mathbf{F})$ contains~$\mathbb{Z}/n\mathbb{Z}$ if and only if $\zeta_n+\zeta_n^{-1} \in \mathbf{F}$.
\end{proposition}

From this proposition we immediately obtain the following result.

\begin{lemma}
\label{l10}
Let $\mathbf{F}$ be a~field of characteristic~$2$ not containing nontrivial cube roots of unity. Then the group $\mathrm{PGL}_2(\mathbf{F})$ does not contain $\mathbb{Z}/5\mathbb{Z}$.
\end{lemma}

\begin{proof}
Assume the converse. Then $\zeta_5+\zeta_5^4 \in \mathbf{F}$ by Proposition~\ref{p2}. However, as $\mathbf{F}$
 is a~field of characteristic~$2$, we have
$$
(\zeta_5+\zeta_5^4)^3=1,
$$
which contradicts the assumption that $\mathbf{F}$ does not contain nontrivial cube roots of unity.

The lemma is proved.
\end{proof}

Now we recall one useful lemma from~\cite{18}.

\begin{lemma}[{\cite{18}, Lemma~5.15}]
\label{l11}
Let $S$ be a~smooth cubic surface over a~field~$\mathbf{F}$ such that its automorphism group is
either isomorphic to $\mathrm{PSU}_4(\mathbb{F}_2)$, or to the subgroup~\mbox{$(\mathbb{Z}/2\mathbb{Z})^4 \rtimes \mathfrak{A}_5$} in the group $\mathrm{PSU}_4(\mathbb{F}_2)$.
Then the image $\Gamma$ of the absolute Galois group $\mathrm{Gal}(\mathbf{F}^{\mathrm{sep}}/\mathbf{F})$ in $W(\mathrm{E}_6)$ is trivial.
\end{lemma}

\section{Cubic surfaces}
\label{s3}

In this section we introduce the notation related to smooth cubic surfaces and recall the description of their geometry. More details on smooth cubic surfaces can be found in~\cite{7}, Ch.~9, or~\cite{14}, Ch.~4.

Let $S \subset \mathbb{P}^3$ be a~smooth cubic surface over a~field $\mathbf{F}$. Recall that by Theorem~\ref{t4} the lines on $S_{\overline{\mathbf{F}}}$ are defined over $\mathbf{F}^{\mathrm{sep}}$, and we have $\mathrm{Pic}(S_{\overline{\mathbf{F}}}) \simeq \mathrm{Pic}(S_{\mathbf{F}^{\mathrm{sep}}})$. So $S_{\mathbf{F}^{\mathrm{sep}}}$ is a~blowup of six points
$$
P_1, \ P_2, \ P_3, \ P_4, \ P_5, \ P_6
$$
in general position on $\mathbb{P}^2_{\mathbf{F}^{\mathrm{sep}}}$. Recall the definition.

\begin{definition}
\label{d1}
We say that points $P_1,\dots, P_i \in \mathbb{P}^2$ for $i \leqslant 6$ are \textit{in general position} if no three points in this set lie on a~line and no six points in this set lie on a~conic.
\end{definition}

From the above we see that there exists a~birational morphism
$$
\pi \colon S_{\mathbf{F}^{\mathrm{sep}}} \to \mathbb{P}^2_{\mathbf{F}^{\mathrm{sep}}},
$$
which blows down the preimages of the points $P_1,\dots, P_6$. We denote by
$$
E_1, \ E_2, \ E_3, \ E_4, \ E_5, \ E_6
$$
the preimages of $P_1$, $P_2$, $P_3$, $P_4$,~$P_5$ and $P_6$, respectively, under the morphism $\pi$. The exceptional curves are lines on $S_{\mathbf{F}^{\mathrm{sep}}} \subset \mathbb{P}^3_{\mathbf{F}^{\mathrm{sep}}}$. The self-intersection index of any line on a~smooth cubic surface is equal to~$-1$. We denote by $Q_i$, $i \in \{1, \dots, 6\}$, the proper transform under the morphism~$\pi$ of the smooth conic passing through all points~$P_1, \dots, P_6$ except for~$P_i$. For $i,j \in \{1, \dots, 6\}$ such that $i<j$ we denote by $L_{ij}$ the proper transform under~$\pi$ of the line passing through~$P_i$ and $P_j$. The proper transforms~$Q_i$ and~$L_{ij}$, for~$i,j \in \{1, \dots, 6\}$ such that $i<j$, and $E_1, \dots, E_6$ give a~set of $27$ lines on~$S_{\mathbf{F}^{\mathrm{sep}}}$. It is well known that there are exactly~$27$ lines on a~smooth cubic surface over a~separably closed field. Their intersections are:
\begin{gather*}
E_i^2=Q_i^2=L_{ij}^2=-1 \quad \text{for } i \neq j;
\\
E_i \cdot E_j=Q_i \cdot Q_j=0 \quad \text{for } i \neq j;
\\
E_i \cdot Q_j=L_{ij} \cdot E_i=L_{ij} \cdot Q_i=1 \quad \text{for } i \neq j;
\\
L_{ij} \cdot L_{kl}=1, \quad \text{where } i,j,k,l \in \{1,\dots,6\} \text{ are pairwise distinct};
\\
L_{ij} \cdot L_{jk}=0, \quad \text{where } i,j,k \in \{1,\dots,6\}\text{ are pairwise distinct}.
\end{gather*}

Consider the set
$$
\Delta=\{\alpha \in \mathrm{Pic}(S_{\mathbf{F}^{\mathrm{sep}}}) \mid K_{S_{\mathbf{F}^{\mathrm{sep}}}}\cdot \alpha=0 \text{ and } \alpha^2=-2 \}.
$$
This is the root system $\mathrm{E}_6$. The group $W(\mathrm{E}_6)$ acts on the Picard group~$\mathrm{Pic}(S_{\mathbf{F}^{\mathrm{sep}}})$ by orthogonal transformations with respect to the intersection form. This action restricts to the configuration of lines on $S_{\mathbf{F}^{\mathrm{sep}}}$. It fixes $K_{S_{\mathbf{F}^{\mathrm{sep}}}}$ and preserves the cone of effective divisors.

Recall the following important theorem.

\begin{theorem}[{see, for instance,~\cite{7}, Corollary~8.2.40}]
\label{t5}
The automorphism group of a~smooth cubic surface $S$ is isomorphic to a~subgroup in~$W(\mathrm{E}_6)$.
\end{theorem}

\begin{remark}
\label{r3}
Although Theorem~\ref{t5} in~\cite{7} was proved over fields of characteristic zero, exactly the same proof also works in the case of positive characteristic.
\end{remark}

 Let $S$ be a~smooth cubic surface over a~field $\mathbf{F}$. We denote by
$\operatorname{Aut}(\mathrm{Pic}(S_{\mathbf{F}^{\mathrm{sep}}}))$ the group of automorphisms of $\mathrm{Pic}(S_{\mathbf{F}^{\mathrm{sep}}})$ which
preserve the canonical class of~$S_{\mathbf{F}^{\mathrm{sep}}}$ and the
intersection form. The absolute Galois
group $\mathrm{Gal}(\mathbf{F}^{\mathrm{sep}}/\mathbf{F})$ acts on~$\mathrm{Pic}(S_{\mathbf{F}^{\mathrm{sep}}})$ and preserves the canonical
class of~$S$ and the intersection form. It is well known
(see~\cite{14}, Theorem~23.9) that
$\operatorname{Aut}(\mathrm{Pic}(S_{\mathbf{F}^{\mathrm{sep}}}))$ is isomorphic to
the Weyl group~$W(\mathrm{E}_6)$. Thus, by Theorem~23.8 in~\cite{14} we have
the natural homomorphism
$$
\boldsymbol{g}\colon \mathrm{Gal}(\mathbf{F}^{\mathrm{sep}}/\mathbf{F}) \to W(\mathrm{E}_6).
$$

Let $\Gamma$ be the image of the absolute Galois group~$\mathrm{Gal}(\mathbf{F}^{\mathrm{sep}}/\mathbf{F})$ in $W(\mathrm{E}_6)$ under the homomorphism $\boldsymbol{g}$. We remind the reader that $\Gamma$ lies in the centralizer of $\operatorname{Aut}(S)$, since any $\phi \in \operatorname{Aut}(S)$ is defined over~$\mathbf{F}$.

\begin{remark}
\label{r4}
Throughout the paper, by the image of the absolute Galois group in the Weyl group $W(\mathrm{E}_6)$ we mean the image of $\mathrm{Gal}(\mathbf{F}^{\mathrm{sep}}/\mathbf{F})$ under the homomorphism~$\boldsymbol{g}$.
\end{remark}

Recall the definition of the graph of lines on a~projective variety. This is the graph whose vertices correspond to lines on a~projective variety and such that there is an edge between two vertices if and only if their corresponding lines intersect. In what follows we need the following definition.

\begin{definition}
\label{d2}
Let $S$ be a~smooth cubic surface over a~field $\mathbf{F}$. We say that five lines $l_1, \dots, l_5$ on $S$ are \textit{in star position} if their graph is a~star (see Figure~\ref{fig1}).
\end{definition}

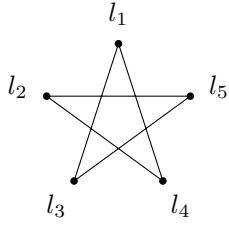
\begin{figure}[H]
\begin{center}

\begin{tikzpicture}[dot/.style={circle,fill,inner sep=1pt}]
\def\r{1}  
\def\n{4}  
\def\Vlabel{{"$l_1$","$l_2$","$l_3$","$l_4$","$l_5$"}}
\foreach \i in {0,...,\n}
\path ({90+\i*360/(\n+1)}:\r) coordinate (V\i) node[dot]{}
+({90+\i*360/(\n+1)}:.4) node{\pgfmathparse{\Vlabel[\i]}\pgfmathresult}
;
\draw (V2)--(V0)--(V3)--(V1)--(V4)--cycle;
\end{tikzpicture}
\end{center}
\caption{Lines in star position}
\label{fig1}
\end{figure}

We go back to conjugacy classes of elements in $W(\mathrm{E}_6)$. Now we regard elements
of $W(\mathrm{E}_6)$ as transformations of the Picard group of a~smooth cubic surface over an algebraically closed field which preserve the intersection form.

\begin{lemma}
\label{l12}
Let $S$ be a~smooth cubic surface over a~separably closed field. Consider an element $g \in W(\mathrm{E}_6)$ acting on the configuration of lines on $S$. The following assertions hold.

{\rm (i)} Assume that $g$ is of order~$5$. Then there are two lines on~$S$ invariant under the action of $g$. Under the action of $g$ there are three invariant quintuples of skew lines on~$S$, and the lines in each quintuple are transitively permuted by $g$. Moreover, under the action of $g$ there are two invariant quintuples of lines in star position on~$S$, and in each quintuple the lines are transitively permuted by~$g$.

{\rm (ii)} Assume that $g$ is of order~$10$. Then there is a~unique pair of skew lines on~$S$ which is invariant under the action of $g$, and the lines in this pair are interchanged by $g$. There is a~unique quintuple of skew lines on $S$ which is invariant under the action of $g$, and the lines in this quintuple are transitively permuted by~$g$. Under the action of $g$ there are two invariant quintuples of lines in star position on~$S$, and the lines in each quintuple are transitively permuted by $g$. Moreover, the ten remaining lines on~$S$ form a~$g$-invariant decuple of lines on $S$, and the lines in this decuple are transitively permuted by $g$.
\end{lemma}

\begin{proof}
Over an algebraically closed field the lemma follows from Table~7.1, co\mbox{lumn~8,} in~\cite{1}. From Theorem~\ref{t4},\,(ii), we obtain the statement for a~separably closed field.

The lemma is proved.
\end{proof}

\begin{lemma}
\label{l13}
Let $S$ be a~smooth cubic surface over a~separably closed field. Let~$l_1$ and $l_2$ be two intersecting lines on $S$. Then there is a~unique line $l_3 \subset S$ that intersects both $l_1$ and $l_2$.
\end{lemma}

\begin{proof}
Consider the hyperplane $T$ in $\mathbb{P}^3$ such that $l_1,l_2 \subset T$. Then
\begin{equation}
\label{eq3.1}
T \cap S=l_1 \cup l_2 \cup l_3
\end{equation}
for some line $l_3 \subset S$. Assume that there is another line $l_4$ that does not lie in $T$ and intersects both $l_1$ and $l_2$. By~\eqref{eq3.1} we get that
\begin{equation}
\label{eq3.2}
-K_S=l_1+l_2+l_3.
\end{equation}
Therefore, on the one hand we obtain $-K_S \cdot l_4=1$, and on the other hand, by~\eqref{eq3.2} we obtain
$$
-K_S \cdot l_4=l_1 \cdot l_4+l_2 \cdot l_4+l_3 \cdot l_4=2+l_3 \cdot l_4 \neq 1.
$$
This contradiction gives us the uniqueness of the line which intersects both $l_1$ and~$l_2$.

The lemma is proved.
\end{proof}

\begin{lemma}
\label{l14}
Let $S$ be a~smooth cubic surface over a~separably closed field. Consider an element $g \in \operatorname{Aut}(S)$ of order~$5$ acting on $S$. Then the two lines invariant under the action of $g$ are skew lines.
\end{lemma}

\begin{proof}
Let $l_1$ and $l_2$ be two lines on $S$ that are invariant under the action of~$g$. Assume that $l_1$ and $l_2$ intersect. Then by Lemma~\ref{l13} there is a~unique line $l_3$ such that $l_3$ intersects both $l_1$ and $l_2$. However, since $l_1$ and $l_2$ are fixed by $g$, $l_3$ is also fixed by $g$. This contradicts Lemma~\ref{l12},\,(i).

The lemma is proved.
\end{proof}

\begin{lemma}
\label{l15}
Let $\mathbf{F}$ be a~separably closed field of characteristic different from~$3$. Let~$S$ be a~smooth cubic surface over $\mathbf{F}$. Let $g$ be an automorphism of $S$ of order~$3$. Denote by $\omega \in \mathbf{F}$ a~nontrivial cube root of unity. Then for the automorphism~$g$ one of the following holds.

{\rm (i)} The automorphism $g$ lies in the conjugacy class $A_2$, and up to conjugation the action on $\mathbb{P}^3$ is given by
$$
\begin{pmatrix}
\omega^2 & 0 & 0 & 0\\
0 & \omega^2 & 0 & 0\\
0 & 0 & \omega & 0\\
0 & 0 & 0 & \omega
\end{pmatrix}.
$$

{\rm (ii)} The automorphism $g$ lies in the conjugacy class $A_2^2$, and up to conjugation the action on $\mathbb{P}^3$ is given by
$$
\begin{pmatrix}
\omega^2 & 0 & 0 & 0\\
0 & \omega & 0 & 0\\
0 & 0 & 1 & 0\\
0 & 0 & 0 & 1
\end{pmatrix}.
$$

{\rm (iii)} The automorphism $g$ lies in the conjugacy class $A_2^3$, and up to conjugation the action on $\mathbb{P}^3$ is given by
$$
\begin{pmatrix}
\omega & 0 & 0 & 0\\
0 & 1 & 0 & 0\\
0 & 0 & 1 & 0\\
0 & 0 & 0 & 1
\end{pmatrix}.
$$
\end{lemma}

\begin{proof}
For an algebraically closed field this lemma follows from Theorem~10.4 in~\cite{8}. Applying Lemma~\ref{l2} we obtain the statement for a~separably closed field.
\end{proof}

\begin{corollary}
\label{c2}
If $g$ is an automorphism of order~$3$ of a~smooth cubic surface over a~field $\mathbf{F}$ of characteristic different from~$3$, such that $\mathbf{F}$ does not contain a~nontrivial cube root of unity, then $g$ does not lie in the conjugacy class $A_2^3$.
\end{corollary}

\begin{proof}
Assume that $g$ lies in the conjugacy class $A_2^3$. By Lemma~\ref{l15},\,(iii), over $\mathbf{F}^{\mathrm{sep}}$ the element $g$ is given up to conjugation by the matrix
$$
A=\begin{pmatrix}
\omega & 0 & 0 & 0\\
0 & 1 & 0 & 0\\
0 & 0 & 1 & 0\\
0 & 0 & 0 & 1
\end{pmatrix}.
$$

Consider the characteristic polynomial of $aA$, where $a \in \overline{\mathbf{F}}^*$. This polynomial is of the form
$$
F(\lambda)=(\lambda-a\omega)(\lambda-a)^3=\lambda^4-a\lambda^3(3+\omega)+a^2\lambda^2(3+\omega)-a^3\lambda(3+\omega)+a^4\omega.
$$
Since $g$ is an automorphism of a~smooth cubic surface over the field $\mathbf{F}$, there is~\mbox{$a \in \overline{\mathbf{F}}^*$} such that all coefficients of $F(\lambda)$ lie in~$\mathbf{F}$. In particular, the ratio of the coefficient of $\lambda^2$ to the square of that of $\lambda^3$ belongs to $\mathbf{F}$:
$$
\frac{a^2(3+\omega)}{a^2(3+\omega)^2}=\frac{1}{3+\omega} \in \mathbf{F}.
$$
But this is not true by assumption. This contradiction completes the proof.
\end{proof}

The contrapositive statement of Corollary~\ref{c2} is as follows.

\begin{corollary}
\label{c3}
Let $S$ be a~smooth cubic surface over a~field $\mathbf{F}$ of characteristic different from~$3$. Assume that there is an element $g \in \operatorname{Aut}(S)$ that lies in the conjugacy class $A^3_2$. Then $\mathbf{F}$ contains a~nontrivial cube root of unity $\omega \in \mathbf{F}^{\mathrm{sep}}$.
\end{corollary}

Now we recall the structure of the automorphism group of the Fermat cubic surface.

\begin{example}
\label{e1}
Consider the Fermat cubic surface $S$ over a~field $\mathbf{F}$ of characteristic different from~$3$. Assume that a~nontrivial cube root of unity $\omega \in \mathbf{F}^{\mathrm{sep}}$ lies in $\mathbf{F}$. It~is not difficult to see that
$$
\operatorname{Aut}(S) \supseteq (\mathbb{Z}/3\mathbb{Z})^3 \rtimes \mathfrak{S}_4,
$$
where $(\mathbb{Z}/3\mathbb{Z})^3$ acts as multiplication by cube roots of unity on the homogeneous coordinates $x$, $y$, $z$ and~$t$ on $\mathbb{P}^3$ and $\mathfrak{S}_4$ acts as a~permutation of the homoge\-neous~coordinates $x$, $y$, $z$ and~$t$. By Theorem~\ref{t1} and Remark~\ref{r1} we have
$$
\operatorname{Aut}(S_{\mathbf{F}^{\mathrm{sep}}}) \simeq
\begin{cases}
 (\mathbb{Z}/3\mathbb{Z})^3 \rtimes \mathfrak{S}_4 &\text{if }\operatorname{char}\mathbf{F} \neq 2,
 \\
\mathrm{PSU}_4(\mathbb{F}_2) &\text{if }\operatorname{char}\mathbf{F} = 2.
\end{cases}
$$
Therefore, $\operatorname{Aut}(S) \simeq (\mathbb{Z}/3\mathbb{Z})^3 \rtimes \mathfrak{S}_4$ if, in addition, the characteristic of $\mathbf{F}$ is different from~$2$.
\end{example}

\begin{lemma}
\label{l16}
Let $G \simeq (\mathbb{Z}/3\mathbb{Z})^3 \rtimes \mathfrak{S}_4$ be isomorphic to a~group acting regularly on the Fermat cubic surface over an algebraically closed field of characteristic different from~$3$. Let $\tau=gh \in G$ be an element of order~$3$ such that $g \in (\mathbb{Z}/3\mathbb{Z})^3$ and $h \in \mathfrak{S}_4$ are nonidentity elements of $G$. Then in some coordinates $g$ is given by
$$
g=\begin{pmatrix}
\omega^a & 0 & 0 & 0\\
0 & \omega^b & 0 & 0\\
0 & 0 & \omega^c & 0\\
0 & 0 & 0 & 1
\end{pmatrix}
$$
and $h=(123)$ is a~permutation of the coordinates, where $a,b,c \in \{0,1,2\}$ are either pairwise different numbers, or all equal to $1$ or~$2$.
\end{lemma}

\begin{proof}
As $\tau$ is an element of order~$3$, we obtain
$$
ghghgh=e,
$$
where $e \in G$ is the identity element. This is equivalent to the equation
\begin{equation}
\label{eq3.3}
g(hgh^{-1})(h^{-1}gh)=e.
\end{equation}

According to Example~\ref{e1} and by the uniqueness up to conjugation of the subgroup~\mbox{$(\mathbb{Z}/3\mathbb{Z})^3 \rtimes \mathfrak{S}_4$}
in the group $\mathrm{PSU}_4(\mathbb{F}_2)$, which follows from Lemma~\ref{l9}, $g$ is given up to conjugation by the element
$$
g=\begin{pmatrix}
\omega^{a_1} & 0 & 0 & 0\\
0 & \omega^{a_2} & 0 & 0\\
0 & 0 & \omega^{a_3} & 0\\
0 & 0 & 0 & \omega^{a_4}
\end{pmatrix}
$$ of $\mathrm{PGL}_4(\mathbf{F})$
for a~separably closed field $\mathbf{F}$ of characteristic different from $3$, where
$$
a_1, a_2, a_3, a_4 \in \{0,1,2\}
$$
are not all equal. Let $\sigma$ be the permutation of the set $\{1,2,3,4\}$ corresponding to~$h$. Then the element
$$
g(hgh^{-1})(h^{-1}gh)
$$
is given by
\begin{equation}
\label{eq3.4}
\begin{pmatrix}
\omega^{a_1+a_{\sigma(1)}+a_{\sigma^2(1)}} & 0 & 0 & 0\\
0 & \omega^{a_2+a_{\sigma(2)}+a_{\sigma^2(2)}} & 0 & 0\\
0 & 0 & \omega^{a_3+a_{\sigma(3)}+a_{\sigma^2(3)}} & 0\\
0 & 0 & 0 & \omega^{a_4+a_{\sigma(4)}+a_{\sigma^2(4)}}
\end{pmatrix}.
\end{equation}

According to~\eqref{eq3.3}, this is equal to the identity element. By our assumptions the element $\sigma$ is of order~$3$. So up to conjugation we have $\sigma=(123)$. This means that
$$
 \omega^{a_4+a_{\sigma(4)}+a_{\sigma^2(4)}}=1.
$$ Thus,
since the element~\eqref{eq3.4} is the identity element, we get that
$$
 \omega^{a_i+a_{\sigma(i)}+a_{\sigma^2(i)}}=1
$$
for $i=1,2,3$. Therefore, $a_1, a_2, a_3 \in \{0,1,2\}$ are either pairwise different, or all equal to $1$ or~$2$.

The lemma is proved.
\end{proof}

\begin{lemma}
\label{l17}
Let $G \simeq (\mathbb{Z}/3\mathbb{Z})^3 \rtimes \mathfrak{S}_4$ be a~group isomorphic to the group acting regularly on the Fermat cubic surface over an algebraically closed field of characteristic different from~$3$. Let
$$
\mathrm{pr} \colon G \to \mathfrak{S}_4
$$
denote the projection homomorphism onto $\mathfrak{S}_4$. Let us consider the group $G$ as a~subgroup of $W(\mathrm{E}_6)$. Then all elements of $G$ from the conjugacy class~$A_2^3$ lie in~$\mathrm{Ker}(\mathrm{pr})$.
\end{lemma}

\begin{proof}
It is not difficult to see from Lemma~\ref{l15},\,(iii), and Example~\ref{e1} that there are elements of $\mathrm{Ker}(\mathrm{pr})$ lying in~$A_2^3$. Let us prove that no elements of $G$ belonging to~$A_2^3$ can occur out of $\mathrm{Ker}(\mathrm{pr})$. Assume that there is an element $\tau=gh$ in $G$, where~\mbox{$g \in (\mathbb{Z}/3\mathbb{Z})^3$} and $h \in \mathfrak{S}_4$ are nonidentity elements such that $\tau$ lies in $A_2^3$. Thus,~$g$ and $h$ satisfy the assumptions of Lemma~\ref{l16}. So up to conjugation~$g$ is given by
$$
\begin{pmatrix}
\omega^a & 0 & 0 & 0\\
0 & \omega^b & 0 & 0\\
0 & 0 & \omega^c & 0\\
0 & 0 & 0 & 1
\end{pmatrix},
$$
where $a,b,c \in \{0,1,2\}$ are either pairwise different, or all equal to $1$ or~$2$, and $h$~acts on the diagonal elements of $g$ as the permutation $(123)$. It is not difficult to see from Lemma~\ref{l15},\,(iii), that such elements~$gh$ do not lie in the conjugacy class~$A_2^3$.

The lemma is proved.
\end{proof}

\begin{lemma}
\label{l18}
Let $G \simeq (\mathbb{Z}/3\mathbb{Z})^3 \rtimes \mathfrak{S}_4$ be isomorphic to a~group acting regularly on the Fermat cubic surface over an algebraically closed field of characteristic different from~$3$. Let $g$ be a~nonidentity element of the normal subgroup~$(\mathbb{Z}/3\mathbb{Z})^3$ of $G$. Then the centralizer of $g$ in $G$ is isomorphic to~$(\mathbb{Z}/3\mathbb{Z})^3 \rtimes D$, where $D$ is isomorphic either to $\mathbb{Z}/2\mathbb{Z}$, $(\mathbb{Z}/2\mathbb{Z})^2$, or to~$\mathfrak{S}_3$.
\end{lemma}

\begin{proof}
 By Example~\ref{e1} and the uniqueness of the subgroup $(\mathbb{Z}/3\mathbb{Z})^3 \rtimes \mathfrak{S}_4$ in $\mathrm{PSU}_4(\mathbb{F}_2)$ up to conjugation, which follows from Lemma~\ref{l9}, the normal subgroup $(\mathbb{Z}/3\mathbb{Z})^3$ of~$G$ consists of the elements
\begin{equation}
\label{eq3.5}
\begin{pmatrix}
\omega^a & 0 & 0 & 0\\
0 & \omega^b & 0 & 0\\
0 & 0 & \omega^c & 0\\
0 & 0 & 0 & \omega^d
\end{pmatrix} \in \mathrm{PGL}_4(\mathbf{F}),
\end{equation}
where $\omega \in \mathbf{F}$ is a~nontrivial cube root of unity in a~field $\mathbf{F}$ of characteristic different from $3$ and $a,b,c,d \in \{0,1,2\}$. The subgroup~$\mathfrak{S}_4$ acts on $(\mathbb{Z}/3\mathbb{Z})^3$ as a~permutation of~the diagonal elements in the representation of elements of $(\mathbb{Z}/3\mathbb{Z})^3$ as elements~of the form~\eqref{eq3.5} in $\mathrm{PGL}_4(\mathbf{F})$.

Now consider all possible centralizers of a~nonidentity element
$$
g \in (\mathbb{Z}/3\mathbb{Z})^3
$$
of the group~$G$. Up to conjugation we have the following three cases.
\begin{itemize}
\item Let $g$ be an element such that $a=b$, $c \neq d$ and $c,d \neq a$. Then the centralizer of $g$ in $G$ is $\mathrm{C}_G(g) \simeq (\mathbb{Z}/3\mathbb{Z})^3 \rtimes \mathbb{Z}/2\mathbb{Z}$.

\item Let $g$ be an element such that $a=b$, $c=d$ and $c \neq a$. Then the centralizer of $g$ in $G$ is $\mathrm{C}_G(g) \simeq (\mathbb{Z}/3\mathbb{Z})^3 \rtimes (\mathbb{Z}/2\mathbb{Z})^2$.

\item Let $g$ be an element such that $a=b=c$ and $d \neq a$. Then the centralizer of~$g$ in $G$ is $\mathrm{C}_G(g) \simeq (\mathbb{Z}/3\mathbb{Z})^3 \rtimes \mathfrak{S}_3$.
\end{itemize}

The lemma is proved.
\end{proof}

\section{The group $\mathfrak{S}_5$ and the Clebsch cubic surface}
\label{s4}

In this section we consider the group $\mathfrak{S}_5$ as the automorphism group of a~smooth cubic surface over an arbitrary field $\mathbf{F}$ of characteristic different from~$5$. First of~all, recall the structure of the automorphism group of the Clebsch cubic surface.

\begin{example}
\label{e2}
It is obvious that the automorphism group of the Clebsch cubic surface over an arbitrary field $\mathbf{F}$ of characteristic different from $5$ contains the group~$\mathfrak{S}_5$ which acts on the projective space~$\mathbb{P}^4$ by permutations of the homogeneous coordinates~$x$,~$y$,~$z$,~$t$ and~$w$. Note that in this case, by Lemma~\ref{l7} the image of the absolute Galois group $\mathrm{Gal}(\mathbf{F}^{\mathrm{sep}}/\mathbf{F})$ in~$W(\mathrm{E}_6)$ lies in a~group isomorphic to~$\mathbb{Z}/2\mathbb{Z}$.
\end{example}

\begin{remark}
\label{r5}
Let $P_1$, $P_2$, $P_3$ and~$P_4$ be four points in general position on~$\mathbb{P}^2$. Then this quadruple of points is unique up to a~projective transformation. 
\end{remark}

\begin{lemma}
\label{l19}
Let $F$ be a~smooth del~Pezzo surface of degree~$5$ over a~field $\mathbf{F}$. Assume that all lines on $F_{\mathbf{F}^{\mathrm{sep}}}$ are defined over~$\mathbf{F}$. Then such an $F$ is unique up to isomorphism.
\end{lemma}

\begin{proof}
As all lines on $F_{\mathbf{F}^{\mathrm{sep}}}$ are defined over $\mathbf{F}$, we can blow down four skew lines and obtain~$\mathbb{P}^2$. The surface $F$ is a~blowup of four points in general position on $\mathbb{P}^2$. By Remark~\ref{r5} four points in general position on $\mathbb{P}^2$ are defined uniquely up to an automorphism of $\mathbb{P}^2$. 
Therefore, a~del~Pezzo surface of degree~$5$ such that all lines on it over $\mathbf{F}^{\mathrm{sep}}$ are defined over $\mathbf{F}$ is unique up to isomorphism.

The lemma is proved.
\end{proof}

\begin{example}
\label{e3}
Let $\mathbf{F}$ be a~separably closed field of characteristic different from~$5$. Let us construct a~smooth del~Pezzo surface of degree~$5$ with a~nontrivial regular action of~$\mathbb{Z}/5\mathbb{Z}$ over the field $\mathbf{F}$. Let $C \subset \mathbb{P}^2$ be a~smooth conic with a~nontrivial action of~$\mathbb{Z}/5\mathbb{Z}$. The action of~$\mathbb{Z}/5\mathbb{Z}$ on $C$ extends to an action of $\mathbb{P}^2$. The group~$\mathbb{Z}/5\mathbb{Z}$ fixes two points~$P_1$ and~$P_2$ on~$C$. Therefore, it fixes the line $\mathcal{L} \subset \mathbb{P}^2$ passing through $P_1$ and~$P_2$. Note that $\mathbb{Z}/5\mathbb{Z}$ acts on $\mathcal{L}$ nontrivially, because otherwise~$\mathbb{Z}/5\mathbb{Z}$ is generated up to conjugation by an element
of the form
$$
\begin{pmatrix}
\zeta_5 & 0 & 0\\
0 & 1 & 0\\
0 & 0 & 1
\end{pmatrix} \in \mathrm{PGL}_3(\mathbf{F}),
$$
where $\zeta_5 \in \mathbf{F}$ is a~nontrivial fifth root of unity.
In this case there are no smooth conics on~$\mathbb{P}^2$ fixed by $\mathbb{Z}/5\mathbb{Z}$. Thus, there is a~point~$P_3$ outside~$C$ that is fixed by~$\mathbb{Z}/5\mathbb{Z}$.

Let us blow up an orbit of five points on $C$ under the action of $\mathbb{Z}/5\mathbb{Z}$. Then we obtain a~$\mathbb{Z}/5\mathbb{Z}$-equivariant morphism
$$
p \colon X \to \mathbb{P}^2
$$
such that $p^{-1}C$ is a~$(-1)$-curve. So we can $\mathbb{Z}/5\mathbb{Z}$-equivariantly blow down $p^{-1}C$ and obtain the smooth del~Pezzo surface of degree~$5$ with two fixed points.
\end{example}

Now we formulate a~direct consequence of Theorem~1.5 in~\cite{19}.

\begin{lemma}
\label{l20}
Let $F$ be a~smooth del~Pezzo surface of degree~$5$ over a~field $\mathbf{F}$ such that all lines on $F_{\mathbf{F}^{\mathrm{sep}}}$ are defined over $\mathbf{F}$. Then
$$
\operatorname{Aut}(F) \simeq \mathfrak{S}_5.
$$
\end{lemma}

\begin{corollary}
\label{c4}
Let $F$ be a~smooth del~Pezzo surface of degree~$5$ with a~nontrivial action of $\mathbb{Z}/5\mathbb{Z}$ over an arbitrary field $\mathbf{F}$ of characteristic different from~$5$. Then this action fixes precisely two points on~$F_{\mathbf{F}^{\mathrm{sep}}}$.
\end{corollary}

\begin{proof}
By Lemma~\ref{l19}, over $\mathbf{F}^{\mathrm{sep}}$ a~smooth del~Pezzo surface of degree~$5$ is unique up to isomorphism. By Lemma~\ref{l20} we have
$$
\operatorname{Aut}(F_{\mathbf{F}^{\mathrm{sep}}}) \simeq \mathfrak{S}_5.
$$

Therefore, as there is a~unique element of order~$5$ in $\mathfrak{S}_5$ up to conjugation, 
 Example~\ref{e3} shows that~$\mathbb{Z}/5\mathbb{Z}$ fixes precisely two points on $F_{\mathbf{F}^{\mathrm{sep}}}$.

The corollary is proved.
\end{proof}

\begin{lemma}
\label{l21}
Let $F$ be a~smooth del~Pezzo surface of degree~$5$ over a~field $\mathbf{F}$ of characteristic different from~$5$. Assume that all lines on $F_{\mathbf{F}^{\mathrm{sep}}}$ are defined over $\mathbf{F}$. Then such a~del~Pezzo surface is unique up to isomorphism. Moreover, $F$ carries a~nontrivial action of $\mathbb{Z}/5\mathbb{Z}$, and such an action is unique up to conjugation.
\end{lemma}

\begin{proof}
By Lemma~\ref{l19} such an~$F$ is unique up to isomorphism. By Lemma~\ref{l20} the automorphism group of $F$ is isomorphic to~$\mathfrak{S}_5$. So there is a~unique action of $\mathbb{Z}/5\mathbb{Z}$ on~$F$ up to conjugation.

The lemma is proved.
\end{proof}

\begin{lemma}
\label{l22}
Let $S$ be a~smooth cubic surface over a~field $\mathbf{F}$ with a~regular action of~$G \simeq\mathbb{Z}/5\mathbb{Z}$ such that the image $\Gamma$ of the absolute Galois group $\mathrm{Gal}(\mathbf{F}^{\mathrm{sep}}/\mathbf{F})$ in the group $W(\mathrm{E}_6)$ lies in a~group isomorphic to~$\mathbb{Z}/2\mathbb{Z}$. Then the $G$-equivariant blowdown of two skew lines
$$
\pi \colon S \to F
$$
gives a~del~Pezzo surface $F$ of degree~$5$ such that any line on $F_{\mathbf{F}^{\mathrm{sep}}}$ is defined over~$\mathbf{F}$.
\end{lemma}

\begin{proof}
By Lemmas~\ref{l12},\,(i), and~\ref{l14} there are two skew lines 
 $l_1$ and~$l_2$ on $S$
 that are fixed by~$G$. Suppose that $\Gamma$ is a~trivial group. Then their blowdown produces a~$G$-equivariant morphism from~$S$ to a~del~Pezzo surface $F$ of degree~$5$ such that any line on $F_{\mathbf{F}^{\mathrm{sep}}}$ is defined over~$\mathbf{F}$.

Now assume that $\Gamma \simeq \mathbb{Z}/2\mathbb{Z}$. Let $g$ be a~generator of $G$ and $\tau$ be a~generator of~$\Gamma$. Then $g$ and $\tau$ commute. So the element $\sigma=g\tau$ is of order~$10$. By Lemma~\ref{l5},\,(iii), the element $\tau$ lies in the conjugacy class~$A_1$. By Lemma~\ref{l12} the lines $l_1$ and $l_2$ are interchanged by $\sigma$. This means that they form a~$\Gamma$-orbit. Hence we can blowdown this $\Gamma$-orbit and obtain a~del~Pezzo surface of degree~$5$:
$$
\pi \colon S \to F.
$$
Since $l_1$ and $l_2$ are fixed by $G$, $\pi$ is a~$G$-equivariant blowdown.

By the properties of blowups we have
$$
\mathrm{Pic}(S_{\mathbf{F}^{\mathrm{sep}}})=\mathrm{Pic}(F_{\mathbf{F}^{\mathrm{sep}}}) \oplus \mathbb{Z}\langle l_1, l_2 \rangle.
$$
By Lemma~\ref{l3} the eigenvalues of the action of $\tau$ on $\mathrm{Pic}(S_{\mathbf{F}^{\mathrm{sep}}})$ are
$$
1,\ 1,\ 1,\ 1,\ 1,\ 1,\ -1.
$$
 Moreover, the element $\tau$ acts nontrivially on the sublattice $\mathbb{Z}\langle l_1, l_2 \rangle$ because $\{l_1, l_2\}$ forms a~$\Gamma$-orbit of length~$2$. This means that $\tau$ acts trivially on $\mathrm{Pic}(F_{\mathbf{F}^{\mathrm{sep}}})$. \mbox{Therefore}, all lines on $F_{\mathbf{F}^{\mathrm{sep}}}$ are defined over~$\mathbf{F}$.

The lemma is proved.
\end{proof}

\begin{lemma}
\label{l23}
Let $S$ be a~smooth cubic surface over a~field $\mathbf{F}$ of characteristic different from~$5$ such that~$\operatorname{Aut}(S) \supseteq \mathbb{Z}/5\mathbb{Z}$ and the image $\Gamma$ of the absolute Galois group~$\mathrm{Gal}(\mathbf{F}^{\mathrm{sep}}/\mathbf{F})$ in $W(\mathrm{E}_6)$ lies in the group isomorphic to $\mathbb{Z}/2\mathbb{Z}$. Then $S$ is isomorphic to the Clebsch cubic surface.
\end{lemma}

\begin{proof}
Consider an element $g \in \operatorname{Aut}(S)$ of order~$5$ and the group $\mathbb{Z}/5\mathbb{Z}$ generated by this element. By Lemma~\ref{l22} there is a~$\mathbb{Z}/5\mathbb{Z}$-equivariant blowdown of~$S$ to a~del~Pezzo surface~$F$ of degree~$5$
$$
\pi \colon S \to F
$$
such that all lines on $F_{\mathbf{F}^{\mathrm{sep}}}$ are defined over $\mathbf{F}$. By Lemma~\ref{l21} this $F$ is unique up to isomorphism and there is a~nontrivial regular action of $\mathbb{Z}/5\mathbb{Z}$ on $F$, which is unique up to conjugation. By Corollary~\ref{c4} the number of fixed points of the action of~$\mathbb{Z}/5\mathbb{Z}$ on~$F$ over $\mathbf{F}^{\mathrm{sep}}$ is equal to~$2$. So the set of two points on $F$ which are fixed by a~regular action of $\mathbb{Z}/5\mathbb{Z}$ is unique up to an automorphism since the regular action of $\mathbb{Z}/5\mathbb{Z}$ is unique up to conjugation. Therefore, a~$\mathbb{Z}/5\mathbb{Z}$-equivariant blowup of two points is unique up to isomorphism. So we obtain the uniqueness of a~smooth cubic surface with the required properties. By Example~\ref{e2} the group~$\mathbb{Z}/5\mathbb{Z}$ acts regularly on the Clebsch cubic surface. This means that $S$ is isomorphic to the Clebsch cubic surface.

The lemma is proved.
\end{proof}

Now we present a~direct corollary of this lemma, which we need in what follows for the proof of the uniqueness of a~smooth cubic surface with a~regular action of~$\mathfrak{S}_5$ over a~field of characteristic different from~$5$.

\begin{corollary}
\label{c5}
Let $\mathbf{F}$ be a~field of characteristic different from $5$. Let $S$ be a~smooth cubic surface over $\mathbf{F}$ such that $\operatorname{Aut}(S) \supseteq \mathfrak{S}_5$. Then $S$ is isomorphic to the Clebsch~cubic surface.
\end{corollary}

\begin{proof}
Since $\operatorname{Aut}(S) \supseteq \mathfrak{S}_5$, by Lemma~\ref{l7} the image $\Gamma$ of the absolute Galois group~$\mathrm{Gal}(\mathbf{F}^{\mathrm{sep}}/\mathbf{F})$ in $W(\mathrm{E}_6)$ lies in a~group isomorphic to $\mathbb{Z}/2\mathbb{Z}$. So we can apply Lemma~\ref{l23} and see that $S$
is isomorphic to the Clebsch cubic surface.

The corollary is proved.
\end{proof}

Before we state some consequences of Lemma~\ref{l23} for characteristic~$2$, we formulate the following two lemmas.

\begin{lemma}
\label{l24}
Let $\mathbf{F}$ be a~field of characteristic~$2$. If there is a~nontrivial cube root of unity in $\mathbf{F}$, then $\mathbb{F}_4 \subset \mathbf{F}$.
\end{lemma}

\begin{proof}
Since the characteristic of $\mathbf{F}$ is equal to~$2$, we get that $\mathbb{F}_2 \subset \mathbf{F}$. As a~nontrivial cube root of unity $\omega \in \mathbf{F}^{\mathrm{sep}}$ lies in $\mathbf{F}$, we obtain
$$
\mathbf{F} \supset \mathbb{F}_2(\omega)=\mathbb{F}_2[t]/(t^2+t+1)=\mathbb{F}_4.
$$

The lemma is proved.
\end{proof}

\begin{lemma}[{see~\cite{18}, Lemmas~5.15 and~5.17}]
\label{l25}
Let $\mathbf{F}$ be a~finite field of characteristic~$2$. Let $S$ be a~smooth cubic surface over $\mathbf{F}$.

{\rm (i)} If\/ $\mathbf{F}=\mathbb{F}_{4^k}$ for some $k \in \mathbb{N}$ and $\operatorname{Aut}(S) \simeq \mathrm{PSU}_4(\mathbb{F}_2)$, then the image in~$W(\mathrm{E}_6)$ of the absolute Galois group $\mathrm{Gal}(\mathbf{F}^{\mathrm{sep}}/\mathbf{F})$ is trivial.

{\rm (ii)} If\/ $\mathbf{F}=\mathbb{F}_{2^{2k+1}}$ for some $k \in \mathbb{Z}_{\geqslant 0}$ and $\operatorname{Aut}(S) \simeq \mathfrak{S}_6$, then the image in $W(\mathrm{E}_6)$ of the absolute Galois group $\mathrm{Gal}(\mathbf{F}^{\mathrm{sep}}/\mathbf{F})$ is isomorphic to~$\mathbb{Z}/2\mathbb{Z}$.
\end{lemma}

\begin{corollary}
\label{c6}
Let $\mathbf{F}$ be a~field of characteristic~$2$. Then the following hold.

{\rm (i)} If\/ $\mathbf{F}$ contains a~nontrivial cube root of unity, then the Fermat cubic surface is isomorphic to the Clebsch cubic surface.

{\rm (ii)} The cubic surface~\eqref{eq1.2} is isomorphic to the Clebsch cubic surface.
\end{corollary}

\begin{proof}
We prove assertion~(i). Let $S$ be the Fermat cubic surface over $\mathbb{F}_4$. By Theorem~\ref{t3},\,(i), we have $\operatorname{Aut}(S) \simeq \mathrm{PSU}_4(\mathbb{F}_2)$, so
$$
\operatorname{Aut}(S) \supset \mathbb{Z}/5\mathbb{Z}.
$$
By Lemma~\ref{l25},\,(i), the image $\Gamma$ of the absolute Galois group $\mathrm{Gal}(\overline{\mathbb{F}}_2/\mathbb{F}_4)$ in $W(\mathrm{E}_6)$~is tri\-vial. Hence by Lemma~\ref{l23} the cubic surface $S$ is isomorphic to the Clebsch~cubic sur\-face over~$\mathbb{F}_4$. By Lemma~\ref{l24} the field $\mathbf{F}$ contains $\mathbb{F}_4$. Therefore,~the~Fermat cubic~surface is isomorphic to the Clebsch cubic surface over any extension $\mathbf{F}$ of $\mathbb{F}_4$.

Let us prove~(ii). Let $S$ be the cubic surface~\eqref{eq1.2} over $\mathbb{F}_2$. By Theorem~\ref{t2},\,(ii), we have $\operatorname{Aut}(S) \simeq \mathfrak{S}_6$, so
$$
\operatorname{Aut}(S) \supset \mathbb{Z}/5\mathbb{Z}.
$$
By Lemma~\ref{l25},\,(ii), the image~$\Gamma$ of the absolute Galois group $\mathrm{Gal}(\overline{\mathbb{F}}_2/\mathbb{F}_2)$ in $W(\mathrm{E}_6)$ is isomorphic to $\mathbb{Z}/2\mathbb{Z}$. So by Lemma~\ref{l23} the cubic surface~$S$ is isomorphic to the~Clebsch cubic surface over~$\mathbb{F}_2$. Therefore, the surface~\eqref{eq1.2} is isomorphic to the Clebsch cubic surface over any extension $\mathbf{F}$ of $\mathbb{F}_2$.

The corollary is proved.
\end{proof}

\section{Characteristic $3$}
\label{s5}

In this section we study smooth cubic surfaces with the largest automorphism group over fields of characteristic~$3$. First of all, we formulate obvious lemmas about fields of characteristic~$3$.

\begin{lemma}
\label{l26}
Let $\mathbf{F}$ be a~field of characteristic~$3$. Then $\mathbf{F}$ contains a~primitive fourth root of unity if and only if\/ $\mathbf{F}$ contains the subfield $\mathbb{F}_9$.
\end{lemma}

\begin{proof}
Since the characteristic of $\mathbf{F}$ is equal to~$3$, we get that $\mathbb{F}_3 \subset \mathbf{F}$. Assume that a~primitive fourth root of unity lies in $\mathbf{F}$. We denote it by $\mathbf{i}$. Then we have
$$
\mathbf{F} \supset \mathbb{F}_3(\mathbf{i})=\mathbb{F}_3[t]/(t^2+1)=\mathbb{F}_9.
$$
Assume that $\mathbf{F} \supset \mathbb{F}_9$. Then there are elements of $\mathbf{F}$ satisfying $t^8=1$. In particular, there are elements of $\mathbf{F}$ satisfying $t^4=1$. This means that there is a~primitive fourth root of unity in~$\mathbf{F}$.

The lemma is proved.
\end{proof}

\begin{lemma}
\label{l27}
Let $\mathbf{F}$ be a~field of characteristic~$3$. Then $\mathbf{F}$ contains a~primitive fourth root of unity if and only if\/ $\mathbf{F}$ contains a~primitive eighth root of unity.
\end{lemma}

\begin{proof}
If $\mathbf{F}$ contains a~primitive fourth root of unity, then by Lemma~\ref{l26} we get the inclusion~\mbox{$\mathbf{F} \supset \mathbb{F}_9$.} Therefore, there are elements in $\mathbf{F}$ which satisfy
$$
t^8=1
\quad \text{and}\quad t^i \neq 1
\quad \text{for any integer } 1 \leqslant i \leqslant 7.
$$
This means that there is a~primitive eighth root of unity in $\mathbf{F}$. If $\mathbf{F}$ does not contain a~primitive fourth root of unity, then it is obvious that it does not contain a~primitive eighth root of unity either.

The lemma is proved.
\end{proof}

Let $\mathbf{F}$ be a~field of characteristic~$3$. We divide this section into two parts: in the first part we consider a~field $\mathbf{F}$ which contains a~primitive fourth root of unity and in the second part we consider a~field $\mathbf{F}$ not containing primitive fourth roots of unity.

We start from the case when $\mathbf{F}$ contains a~primitive fourth root of unity. By Theorem~\ref{t1},\,$(\mathrm{ii})'$, the automorphism group of a~smooth cubic surface of maximum order over an algebraically closed field of characteristic~$3$ is isomorphic to~\mbox{$\mathcal{H}_3(\mathbb{F}_3) \rtimes \mathbb{Z}/8\mathbb{Z}$.}
Let us prove that this group is also realized as the automorphism group of the surface~\eqref{eq1.1} over $\mathbf{F}$.

\begin{lemma}
\label{l28}
Let $\mathbf{F}$ be a~field of characteristic~$3$ containing a~primitive fourth root of unity. Then the automorphism group of the smooth cubic surface~\eqref{eq1.1} over~$\mathbf{F}$ is isomorphic to $\mathcal{H}_3(\mathbb{F}_3) \rtimes \mathbb{Z}/8\mathbb{Z}$.
\end{lemma}

\begin{proof}
By Lemma~\ref{l27} a~primitive eighth root of unity
$\zeta_8 \in \mathbf{F}^{\mathrm{sep}}$
lies in $\mathbf{F}$. Consider the element of order~$8$
$$
h=\begin{pmatrix}
\zeta_8^6 & 0 & 0 & 0\\
0 & \zeta_8 & 0 & 0\\
0 & 0 & \zeta_8^4 & 0\\
0 & 0 & 0 & 1
\end{pmatrix} \in \mathrm{PGL}_4(\mathbf{F}).
$$

Consider the group $G$ of elements
\begin{equation}
\label{eq5.1}
\begin{pmatrix}
1 & \alpha^3 & -\alpha^6 & 0\\
0 & 1 & \alpha^3 & 0\\
0 & 0 & 1 & 0\\
0 & \alpha & c & 1
\end{pmatrix} \in \mathrm{PGL}_4(\overline{\mathbf{F}}),
\end{equation}
where $\alpha$ and $c$ lie in $\overline{\mathbf{F}}$ and satisfy the equations
\begin{equation}
\alpha^9-\alpha=0 \label{eq5.2}
\end{equation}
and
\begin{equation}
c^3+c-\alpha^4=0. \label{eq5.3}
\end{equation}

We can see that $\alpha, c \in \mathbf{F}$. Indeed, by Lemma~\ref{l26} we have $\alpha \in \mathbf{F}$. Moreover, by the same lemma the primitive fourth root of unity $\mathbf{i} \in \mathbf{F}^{\mathrm{sep}}$ lies in $\mathbf{F}$. If $\alpha=0$, then equation~\eqref{eq5.3} has three solutions, $0$ and $\pm \mathbf{i}$. If $\alpha \neq 0$, then from~\eqref{eq5.2} we get that $\alpha^4=\pm 1$. Therefore,~\eqref{eq5.3} can be written as $c^3+c \pm 1=0$. The roots of the polynomial $c^3+c+1$ are $1$ and $1 \pm \mathbf{i} \in \mathbf{F}$, and the roots of $c^3+c-1$ are $-1$ and~\mbox{$-1 \pm \mathbf{i} \in \mathbf{F}$.} Hence the elements of $G$ lie in~$\mathrm{PGL}_4(\mathbf{F})$.

We denote by $(\alpha,c)$ the element~\eqref{eq5.1} of the group $G$. It is not difficult to see the following relation among elements of $G$:
\begin{equation}
\label{eq5.4}
(\alpha_1, c_1) \cdot (\alpha_2,c_2)=(\alpha_1+\alpha_2, \alpha_1\alpha_2^3+c_1+c_2).
\end{equation}

We can check by straightforward computations that $G$ is of order~$27$, and by~\eqref{eq5.4} any nonidentity element of $G$ is of order~$3$ and $G$ is nonabelian. Hence $G \simeq \mathcal{H}_3(\mathbb{F}_3)$.

The smooth cubic surface~\eqref{eq1.1} is invariant under the action of $h$ and $G$ on~$\mathbb{P}^3$ with homogeneous coordinates $x$, $y$, $z$ and~$t$. Therefore, the automorphism group of the surface~\eqref{eq1.1} contains a~subgroup of order~$216$. By Theorem~\ref{t1},\,$(\mathrm{ii})'$, and Remark~\ref{r1}, over the field~$\mathbf{F}^{\mathrm{sep}}$ the automorphism group of the surface~\eqref{eq1.1} is isomorphic to~$\mathcal{H}_3(\mathbb{F}_3) \rtimes \mathbb{Z}/8\mathbb{Z}$, which is of order~$216$. This means that the automorphism group of the smooth cubic surface~\eqref{eq1.1} over $\mathbf{F}$ is isomorphic to the group~$\mathcal{H}_3(\mathbb{F}_3) \rtimes \mathbb{Z}/8\mathbb{Z}$.

The lemma is proved.
\end{proof}

\goodbreak

\begin{lemma}
\label{l29}
Let $\mathbf{F}$ be a~field of characteristic different from~$2$. Assume that it contains a~primitive eighth root of unity $\zeta_8 \in \mathbf{F}^{\mathrm{sep}}$. Let $S$ be a~smooth cubic surface over a~field $\mathbf{F}$. Let $H$ be a~cyclic subgroup of order~$8$ in the automorphism group of~$S$. Then up to conjugation $H$ is generated by the element
\begin{equation}
\label{eq5.5}
h=\begin{pmatrix}
\zeta_8^6 & 0 & 0 & 0\\
0 & \zeta_8 & 0 & 0\\
0 & 0 & \zeta_8^4 & 0\\
0 & 0 & 0 & 1
\end{pmatrix} \in \mathrm{PGL}_4(\mathbf{F}).
\end{equation}
\end{lemma}

For an algebraically closed field of characteristic different from~$2$ this lemma follows directly from~\cite{7}, \S\,9.5.1.

\begin{lemma}
\label{l30}
Let $\mathbf{F}$ be a~field of characteristic different from $2$ such that it contains a~primitive eighth root of unity $\zeta_8 \in \mathbf{F}^{\mathrm{sep}}$. Let $S$ be a~smooth cubic surface over $\mathbf{F}$ such that
$$
\operatorname{Aut}(S) \supset \mathbb{Z}/8\mathbb{Z}.
$$
Then $S$ is isomorphic to a~smooth cubic surface defined by the equation
\begin{equation}
\label{eq5.6}
\alpha t^3+tz^2-xy^2+x^2z=0
\end{equation}
for some $\alpha \!\in\! \mathbf{F}^*$, and the element of $\mathrm{PGL}_4(\mathbf{F})$ of the form~\eqref{eq5.5} fixes this \mbox{polynomial}.
\end{lemma}

\begin{proof}
Let $h \in \operatorname{Aut}(S)$ be an automorphism of order~$8$. Then by Lemma~\ref{l29}, up to conjugation the element $h$ has the form~\eqref{eq5.5}. Denote by $x$, $y$, $z$ and~$t$ the homogeneous coordinates on $\mathbb{P}^3$ such that $h$ is of the form~\eqref{eq5.5}. Then the automorphism~$h$ acts on monomials of the third degree by multiplication by eighth roots of unity.
In Table~\ref{t4}, for each monomial of the third degree with respect to the coordinates $x$, $y$, $z$ and~$t$ one can find the coefficient with which $h$ acts on it.

\begin{table}[!htb]
\renewcommand{\arraystretch}{1.2} 
\centering\begin{tabular}[c]{| >{\small}c| >{\small}c|}
\hline
\text{Coefficient} & \text{Monomials}\\
\hline
$1$ & $t^3$, $x^2z$, $y^2x$, $z^2t$ \\
\hline
$\zeta_8$ & $z^2y$, $t^2y$ \\
\hline
$\zeta_8^2$ & $x^3$, $y^2t$, $xzt$ \\
\hline
$\zeta_8^3$ & $y^3$, $xyz$ \\
\hline
$\zeta_8^4$ & $z^3$, $x^2t$, $t^2z$ \\
\hline
$\zeta_8^5$ & $x^2y$, $yzt$ \\
\hline
$\zeta_8^6$ & $y^2z$, $z^2x$, $t^2x$ \\
\hline
$\zeta_8^7$ & $xyt$ \\
\hline
\end{tabular}
\vskip3mm
\caption{Monomials of the third degree with action of $h$}
\label{tab4}
\end{table}

We can see from Table~\ref{t4} that all smooth cubic surfaces fixed by $h$ are isomorphic to some cubic surface defined by the equation
\begin{equation}
\label{eq5.7}
at^3+btz^2+cxy^2+dx^2z=0
\end{equation}
for $a,b,c,d \in \mathbf{F}^*$. Note that if at least one of the coefficients $a$, $b$, $c$ and $d$ is zero, then~\eqref{eq5.7} is singular. By the change of coordinates
$$
x \mapsto x,
\qquad y \mapsto y,
\qquad -\frac{d}{c}z \mapsto z
\quad\text{and}\quad -\frac{bc}{d^2}t \mapsto t
$$ the
smooth cubic surface~\eqref{eq5.7} is transformed into the surface~\eqref{eq5.6} with~\mbox{$\alpha=\frac{ad^6}{b^3c^4}$.}

The lemma is proved.
\end{proof}

\begin{lemma}
\label{l31}
Let $\mathbf{F}$ be a~field of characteristic~$3$ such that it contains a~primitive eighth root of unity $\zeta_8 \in \mathbf{F}^{\mathrm{sep}}$. Let $S$ be a~smooth cubic surface over $\mathbf{F}$ isomorphic to a~smooth cubic surface~\eqref{eq5.6} for some $\alpha \in \mathbf{F}^*$. Assume that all lines on $S_{\mathbf{F}^{\mathrm{sep}}}$ are defined over $\mathbf{F}$. Then $\alpha \in \mathbf{F}^8$.
\end{lemma}

\begin{proof}
Fix $\alpha \in \mathbf{F}^*$ such that $S$ is isomorphic to $S'$ defined by the equation
$$
\alpha t^3+tz^2-xy^2+x^2z=0.
$$

The line $l$ defined by the system of equations
$$
\begin{cases}
x=0,\\
t=0
\end{cases}
$$ in $\mathbb{P}^3$
lies on $S'$. Consider the pencil of hyperplanes passing through $l$ in $\mathbb{P}^3_{\mathbf{F}^{\mathrm{sep}}}$. Any~hyperplane in this pencil is defined by an equation
\begin{equation}
\label{eq5.8}
x=\lambda t
\quad\text{for } \lambda \in \mathbf{F}^{\mathrm{sep}}
\quad\text{or}\quad t=0.
\end{equation}

Let us find all hyperplanes $T$ of the form~\eqref{eq5.8} such that the intersection of $T$ and~$S'$ consists of three lines. If $T$ is defined by $t=0$, then $T \cap S'$ is a~line and a~smooth conic. Assume that $T$ is defined by $x=\lambda t$ for some $ \lambda \in \mathbf{F}^{\mathrm{sep}}$. Then~$T \cap S'$ is defined by
$$
t(\alpha t^2+\lambda^2 tz +z^2-\lambda y^2)=0.
$$
Let
$$
f=\alpha t^2+\lambda^2 tz +z^2-\lambda y^2.
$$
It is sufficient to find all $\lambda$ such that $f$ is a~product of two different linear functions. We can see from direct computations that this occurs if and only if either $\lambda =0$, or
\begin{equation}
\label{eq5.9}
\alpha=\lambda^4.
\end{equation}
Therefore, from~\eqref{eq5.9} we obtain
$$
f=(\lambda^2t-z)^2-\lambda y^2=(\lambda^2 t-z-\mu y)(\lambda^2 t-z+\mu y),
$$
where $\mu \in \mathbf{F}^{\mathrm{sep}}$ and $\mu^2=\lambda$. Since any line on $S'_{\mathbf{F}^{\mathrm{sep}}}$ is defined over $\mathbf{F}$, we see, in particular, that the lines
$$
\begin{cases}
x=\lambda t,\\
\lambda^2 t-z \pm \mu y=0
\end{cases}
$$
are defined over $\mathbf{F}$. This means that $\mu \in \mathbf{F}$ or, equivalently, $\lambda \in \mathbf{F}^2$. Therefore, by~\eqref{eq5.9} we have $\alpha=\gamma^8$ for some $\gamma \in \mathbf{F}^*$.

The lemma is proved.
\end{proof}

\begin{lemma}
\label{l32}
Let $\mathbf{F}$ be a~field of characteristic~$3$ containing a~primitive eighth root of unity $\zeta_8$. Let $S$ be a~smooth cubic surface over $\mathbf{F}$ isomorphic to~\eqref{eq5.6} for some~\mbox{$\alpha \in \mathbf{F}^*$.} Assume that all lines on $S_{\mathbf{F}^{\mathrm{sep}}}$ are defined over~$\mathbf{F}$. Then $S$ is isomorphic to the surface~\eqref{eq1.1}.
\end{lemma}

\begin{proof}
By Lemma~\ref{l31} the surface $S$ is isomorphic~\eqref{eq5.6} for some nonzero $\alpha \in \mathbf{F}^8$. So we have $\alpha = \gamma^8$ for some $\gamma \in \mathbf{F}^*$. By the change of coordinates
$$
x \mapsto \gamma^2 x, \qquad y \mapsto \gamma^3 y, \qquad z \mapsto \gamma^4 z, \qquad t \mapsto t
$$
equation~\eqref{eq5.6} is transformed into~\eqref{eq1.1}, and we obtain the required result.

The lemma is proved.
\end{proof}

\begin{lemma}[{see~\cite{8}, Table~7}]
\label{l33}
Let $S$ be a~smooth cubic surface over a~field $\mathbf{F}$ of characteristic~$3$. Assume that $\operatorname{Aut}(S) \simeq \mathcal{H}_3(\mathbb{F}_3) \rtimes \mathbb{Z}/8\mathbb{Z}$. If $g \in \operatorname{Aut}(S)$ is an element of order~$6$, then it lies in the conjugacy class $E_6(a_2)$. If $g \in \operatorname{Aut}(S)$ is an~element of order~$3$, then it lies either in the conjugacy class $A_2^2$, or in the conjugacy class $A_2^3$.
\end{lemma}

\begin{lemma}
\label{l34}
Let $S$ be a~smooth cubic surface over a~field $\mathbf{F}$ of characteristic~$3$. Assume that $\operatorname{Aut}(S) \simeq \mathcal{H}_3(\mathbb{F}_3) \rtimes \mathbb{Z}/8\mathbb{Z}$. Then the centre of $\operatorname{Aut}(S)$ is trivial.
\end{lemma}

\begin{proof}
Let $Z$ be the centre of the group $\operatorname{Aut}(S)$. Assume that $Z$ is not the trivial group. Then the order of $Z$ is divisible either by~$2$, or by~$3$. Assume that the order of $Z$ is divisible by $2$. This means that there is an element $h \in Z$ of order~$2$. So for any element $g \in \operatorname{Aut}(S)$ of order~$3$ we have the element
\begin{equation}
\label{eq5.10}
g^2h=hg^2=\tau
\end{equation}
of order~$6$, which lies in the conjugacy class $E_6(a_2)$ by Lemma~\ref{l33}. From~\eqref{eq5.10} we immediately get that $g=\tau^2$. In particular, this means that each element of order~$3$ in~$\operatorname{Aut}(S)$ is a~square of some element of $E_6(a_2)$. By Lemma~\ref{l5},\,(i), this means that each element of order~$3$ in $\operatorname{Aut}(S)$ has to lie in the conjugacy class $A_2^3$, which is impossible by Lemma~\ref{l33}.

Assume that the order of $Z$ is divisible by $3$. Then there is an element of order~$24$ in $\operatorname{Aut}(S)$. However, this is impossible by Lemma~\ref{l8}.

The lemma is proved.
\end{proof}

\begin{lemma}
\label{l35}
Let $S$ be a~smooth cubic surface over a~field $\mathbf{F}$ of characteristic~$3$. Assume that $\operatorname{Aut}(S) \simeq \mathcal{H}_3(\mathbb{F}_3) \rtimes \mathbb{Z}/8\mathbb{Z}$. Then the image $\Gamma$ of the absolute Galois group $\mathrm{Gal}(\mathbf{F}^{\mathrm{sep}}/\mathbf{F})$ in $W(\mathrm{E}_6)$ is trivial.
\end{lemma}

\begin{proof}
By Theorem~\ref{t1},\,$(\mathrm{ii})'$, and Theorem~\ref{t4},\,(i), we have
$$
\operatorname{Aut}(S_{\mathbf{F}^{\mathrm{sep}}}) \simeq \mathcal{H}_3(\mathbb{F}_3) \rtimes \mathbb{Z}/8\mathbb{Z},
$$
and $S_{\mathbf{F}^{\mathrm{sep}}}$ is isomorphic to the surface~\eqref{eq1.1}. This means that $\operatorname{Aut}(S)$ is isomorphic to the automorphism group of~\eqref{eq1.1}. Consider the element $g \in \operatorname{Aut}(S)$ of order~$8$. Since $\Gamma$ commutes with any element of~$\operatorname{Aut}(S)$, $\Gamma$ lies in the centralizer $\mathrm{C}_{W(\mathrm{E}_6)}(g)$ of $g$ in~$W(\mathrm{E}_6)$. By Lemma~\ref{l3} there is a~unique conjugacy class of elements of order~$8$ in $W(\mathrm{E}_6)$ and its centralizer~$\mathrm{C}_{W(\mathrm{E}_6)}(g)$ is of order~$8$. Hence $\mathrm{C}_{W(\mathrm{E}_6)}(g)$ is a~cyclic group of order~$8$ generated by~$g$, and we have
$$
\Gamma \subset \mathrm{C}_{W(\mathrm{E}_6)}(g) \subset \operatorname{Aut}(S).
$$
However, by Lemma~\ref{l34} the centre of $\operatorname{Aut}(S)$ is trivial. This means that $\Gamma$ is trivial.

The lemma is proved.
\end{proof}

\begin{lemma}
\label{l36}
Let $\mathbf{F}$ be a~field of characteristic~$3$ containing a~primitive eighth root of unity $\zeta_8 \in \mathbf{F}^{\mathrm{sep}}$. Let $S$ be a~smooth cubic surface over $\mathbf{F}$ such that
$$
\operatorname{Aut}(S) \simeq \mathcal{H}_3(\mathbb{F}_3) \rtimes \mathbb{Z}/8\mathbb{Z}.
$$
Then $S$ is isomorphic to the surface~\eqref{eq1.1}.
\end{lemma}

\begin{proof}
By Lemma~\ref{l30} the smooth cubic surface $S$ is isomorphic to the surface $S'$ defined by equation~\eqref{eq5.6} for some $\alpha \in \mathbf{F}^*$. By Lemma~\ref{l35} the image $\Gamma$ of the Galois group $\mathrm{Gal}(\mathbf{F}^{\mathrm{sep}}/\mathbf{F})$ in the Weyl group $W(\mathrm{E}_6)$ is trivial. This means that all lines on $S'_{\mathbf{F}^{\mathrm{sep}}}$ are defined over $\mathbf{F}$. So by Lemma~\ref{l32} the smooth cubic surface~$S'$ is isomorphic to the surface~\eqref{eq1.1}.

The lemma is proved.
\end{proof}

Now we consider smooth cubic surfaces over a~field $\mathbf{F}$ of characteristic~$3$ containing no primitive fourth roots of unity. Let us formulate some auxiliary lemmas.

\begin{lemma}
\label{l37}
Let $S$ be a~smooth cubic surface over separabely closed field $\mathbf{F}$ of characteristic different from~$2$. Let $h \in \operatorname{Aut}(S)$ be an automorphism of order~$8$. Then $h$ has three fixed points on~$S$. Moreover, the intersection of $S$ with the tangent planes at fixed points is a~cuspidal cubic curve for one of these points, and it is reducible for the other two.
\end{lemma}

\begin{proof}
By Lemmas~\ref{l29} and~\ref{l30}, up to conjugation $h$ acts on $\mathbb{P}^3$ as
$$
\begin{pmatrix}
\zeta_8^6 & 0 & 0 & 0\\
0 & \zeta_8 & 0 & 0\\
0 & 0 & \zeta_8^4 & 0\\
0 & 0 & 0 & 1
\end{pmatrix} \in \mathrm{PGL}_4(\mathbf{F})
$$
and $S$ is defined by
\begin{equation}
\label{eq5.11}
\alpha t^3+tz^2-xy^2+x^2z=0
\end{equation}
for some $\alpha \in \mathbf{F}^*$.
From~the form of $h$ and $S$
we can see that $h$ fixes three points on~$S$, namely,
$$
[1:0:0:0], \quad [0:1:0:0], \quad [0:0:1:0].
$$
The intersection of $S$ with the tangent plane at $[1:0:0:0]$ is defined by
$$
\begin{cases}
z=0,\\
\alpha t^3-xy^2=0,
\end{cases}
$$
which is a~cuspidal cubic curve. The intersection of $S$ with the tangent plane at~$[0:1:0:0]$ is defined by
$$
\begin{cases}
x=0,\\
t(\alpha t^2+z^2)=0,
\end{cases}
$$
which is a~reducible curve. The intersection of $S$ with the tangent plane at the point~\mbox{$[0:0:1:0]$} is defined by
$$
\begin{cases}
t=0,\\
x(xz-y^2)=0,
\end{cases}
$$
which is a~reducible curve.

The lemma is proved.
\end{proof}

\begin{lemma}
\label{l38}
Let $C$ be a~cuspidal cubic curve over a~field $\mathbf{F}$ such that its singular point is $\mathbf{F}$-rational. Then
$$
\operatorname{Aut}(C) \subseteq \mathbf{F} \rtimes \mathbf{F}^*,
$$
where $\mathbf{F} \rtimes \mathbf{F}^*$ is isomorphic to the group of automorphisms of $\mathbb{P}^1_{\mathbf{F}}$ that fix a point
$$
Q \in \mathbb{P}^1(\mathbf{F}).
$$
\end{lemma}

\begin{proof}
Let $P$ be a~singular point on $C$. Consider the normalization $\eta \colon D \to C$, where $D_{\mathbf{F}^{\mathrm{sep}}} \simeq \mathbb{P}^1_{\mathbf{F}^{\mathrm{sep}}}$. Consider the morphism $\eta$ under the extension of scalars
$$
\eta_{\mathbf{F}^{\mathrm{sep}}} \colon \mathbb{P}^1_{\mathbf{F}^{\mathrm{sep}}} \to C_{ \mathbb{P}^1_{\mathbf{F}^{\mathrm{sep}}}}.
$$
From the definition of a~cuspidal curve we can easily see that $\eta^{-1}_{\mathbf{F}^{\mathrm{sep}}}(P)$ is a~point. Since $P \in C(\mathbf{F})$, we get that $\eta^{-1}(P) \in D(\mathbf{F})$. This means that $D \simeq \mathbb{P}^1$. The morphism $\eta$ is an isomorphism outside~$P$. So the automorphism group of~$C$ lies in a~subgroup in $\mathrm{PGL}_2(\mathbf{F})$ consisting of the automorphisms that fix the point $\eta^{-1}(P)$.

The lemma is proved.
\end{proof}

\begin{lemma}
\label{l39}
Let $\mathbf{F}$ be a~field of characteristic different from~$2$ containing no primitive eighth root of unity. Let $S$ be a~smooth cubic surface over~$\mathbf{F}$. Then $\operatorname{Aut}(S)$ does not contain elements of order~$8$.
\end{lemma}

\begin{proof}
Assume the converse. Let $h \in \operatorname{Aut}(S)$ be an element of order~$8$. Then by Lemma~\ref{l37} the element $h$, as considered over $\mathbf{F}^{\mathrm{sep}}$, has three fixed points on $S_{\mathbf{F}^{\mathrm{sep}}}$ and the intersection of $S_{\mathbf{F}^{\mathrm{sep}}}$ with the tangent plane of the fixed points is a~cuspidal cubic curve for only one of these three points. We denote this point by~$P$ and let~\mbox{$C=S \cap T_{P}S$.} Then $P$ is defined over $\mathbf{F}$, and so is~$C$.

By Theorem~3.7 in~\cite{5}, for the group $\mathbb{Z}/8\mathbb{Z}$ generated by~$h$ we obtain
$$
\mathbb{Z}/8\mathbb{Z} \subset \mathrm{GL}(T_P(S)).
$$
This means that $\mathbb{Z}/8\mathbb{Z} \subset \operatorname{Aut}(C)$. By Lemma~\ref{l38} this means that $\mathbb{Z}/8\mathbb{Z} \subset \mathbf{F} \rtimes \mathbf{F}^*$. Consider the projective homomorphism
$$
\mathrm{pr} \colon \mathbf{F} \rtimes \mathbf{F}^* \to \mathbf{F}^*.
$$
 Since the characteristic of $\mathbf{F}$ is different from $2$, we get that
$$
\mathrm{pr}(\mathbb{Z}/8\mathbb{Z}) \simeq \mathbb{Z}/8\mathbb{Z}.
$$
However, this is impossible because $\mathbf{F}$ does not contain primitive eighth roots of unity.

The lemma is proved.
\end{proof}

\begin{lemma}
\label{l40}
Let $\mathbf{F}$ be a~field of characteristic~$3$ not containing primitive eighth roots of unity. Let $S$ be a~smooth cubic surface over $\mathbf{F}$. Then $|{\operatorname{Aut}(S)}| \leqslant 120$. Moreover, equality holds if and only if $\operatorname{Aut}(S) \simeq \mathfrak{S}_5$ and $S$ is isomorphic to the Clebsch cubic surface.
\end{lemma}

\begin{proof}
 By Theorem~\ref{t1},\,$(\mathrm{ii})'$, and Remark~\ref{r1} the largest automorphism group of a~smooth cubic surface over $\mathbf{F}^{\mathrm{sep}}$ is isomorphic to $\mathcal{H}_3(\mathbb{F}_3) \rtimes \mathbb{Z}/8\mathbb{Z}$ and has order~$216$. By Lemma~\ref{l39} an element of order~$8$ does not lie in $\operatorname{Aut}(S)$. This means that if
$$
\operatorname{Aut}(S_{\mathbf{F}^{\mathrm{sep}}}) \simeq \mathcal{H}_3(\mathbb{F}_3) \rtimes \mathbb{Z}/8\mathbb{Z},
$$
then $|{\operatorname{Aut}(S)}| \leqslant 108$.

So let us consider a~smooth cubic surface $S$ over $\mathbf{F}$ such that $\operatorname{Aut}(S_{\mathbf{F}^{\mathrm{sep}}})$ is the next largest automorphism group of a~smooth cubic surface over $\mathbf{F}^{\mathrm{sep}}$. By Theorem~\ref{t1},\,$(\mathrm{ii})''$, this means that $\operatorname{Aut}(S_{\mathbf{F}^{\mathrm{sep}}}) \simeq \mathfrak{S}_5$. So the order of~$\operatorname{Aut}(S)$ is at most~$120$. By Example~\ref{e2} the group $\mathfrak{S}_5$ is realized on the Clebsch cubic surface over~$\mathbf{F}$. By Corollary~\ref{c5} the smooth cubic surface $S$ is unique up to isomorphism and therefore isomorphic to the Clebsch cubic surface.

The lemma is proved.
\end{proof}

\section{The Fermat cubic surface}
\label{s6}

In this section we study the Fermat cubic surface and the group~\mbox{$(\mathbb{Z}/3\mathbb{Z})^3 \rtimes \mathfrak{S}_4$} which acts regularly on a~smooth cubic surface.
 Note that in what follows we regard~\mbox{$(\mathbb{Z}/3\mathbb{Z})^3 \rtimes \mathfrak{S}_4$} as a~group acting regularly 
 on the Fermat cubic surface over a~separably closed field of characteristic different from $3$. This means that, up to conjugation,~$(\mathbb{Z}/3\mathbb{Z})^3$ acts as multiplication by cube roots of unity on~the homogeneous coordinates $x$, $y$, $z$~and~$t$ on~$\mathbb{P}^3$ and $\mathfrak{S}_4$ acts as a~permutation of the~homogeneous coordinates $x$, $y$,~$z$ and~$t$.

\begin{lemma}
\label{l41}
The centre of the group $G \simeq (\mathbb{Z}/3\mathbb{Z})^3 \rtimes \mathfrak{S}_4$ is trivial.
\end{lemma}

\begin{proof}
Assume the converse. Let $(a,b) \in G$ be an element of the centre of~$G$. Then for any $(c,d) \in G$ we have
\begin{equation}
\label{eq6.1}
(a,b) \cdot (c,d)=(a \ {}^b c,bd)=(c\ {}^d a,db)=(c,d)\cdot (a,b),
\end{equation}
where ${}^b{c}$ and ${}^d{a}$ are the results of the action of $\mathfrak{S}_4$ 
 on $(\mathbb{Z}/3\mathbb{Z})^3$. As the centre of~$\mathfrak{S}_4$ is trivial, we see from~\eqref{eq6.1} that $b$ is the identity element. This means that for any~\mbox{$d \in \mathfrak{S}_4$} we have $a={}^d{a}$. By the definition of $G$ \ $a$ is the identity element.

The lemma is proved.
\end{proof}

Now we are ready to study smooth cubic surfaces with a~regular action of the group $(\mathbb{Z}/3\mathbb{Z})^3 \rtimes \mathfrak{S}_4$ over a~field $\mathbf{F}$ of characteristic different from~$3$.

\begin{lemma}
\label{l42}
Let $S$ be a~smooth cubic surface over a~field $\mathbf{F}$ of characteristic different from~$3$. Assume that $\operatorname{Aut}(S) \supseteq (\mathbb{Z}/3\mathbb{Z})^3 \rtimes \mathfrak{S}_4$. Then the image $\Gamma$ of the absolute Galois group $\mathrm{Gal}(\mathbf{F}^{\mathrm{sep}}/\mathbf{F})$ in $W(\mathrm{E}_6)$ is trivial.
\end{lemma}

\begin{proof}
As $\operatorname{Aut}(S)$ contains $(\mathbb{Z}/3\mathbb{Z})^3 \rtimes \mathfrak{S}_4$, there is an element $g$ of order~$9$ in $\operatorname{Aut}(S)$. Since $\Gamma$ commutes with any element of~$\operatorname{Aut}(S)$, $\Gamma$ lies in the centralizer~$\mathrm{C}_{W(\mathrm{E}_6)}(g)$ of $g$ in~$W(\mathrm{E}_6)$. By Lemma~\ref{l3} there is a~unique conjugacy class of elements of order~$9$ in $W(\mathrm{E}_6)$ and its centralizer~$\mathrm{C}_{W(\mathrm{E}_6)}(g)$ is of order~$9$. Hence $\mathrm{C}_{W(\mathrm{E}_6)}(g)$~is a~cyclic group of order~$9$ generated by~$g$, and we have
$$
\Gamma \subset \mathrm{C}_{W(\mathrm{E}_6)}(g) \subset \operatorname{Aut}(S).
$$
However, the centre of $(\mathbb{Z}/3\mathbb{Z})^3 \rtimes \mathfrak{S}_4$ is trivial by Lemma~\ref{l41}. This means that $\Gamma$ is trivial.

The lemma is proved.
\end{proof}

\begin{lemma}
\label{l43}
Let $S$ be a~smooth cubic surface over a~field $\mathbf{F}$ of characteristic different from~$3$. Assume that $\operatorname{Aut}(S) \supseteq (\mathbb{Z}/3\mathbb{Z})^3 \rtimes \mathfrak{S}_4$. Then $\mathbf{F}$ contains a~nontrivial cube root of unity $\omega \in \mathbf{F}^{\mathrm{sep}}$.
\end{lemma}

\begin{proof}
As $\operatorname{Aut}(S)$ contains $ (\mathbb{Z}/3\mathbb{Z})^3 \rtimes \mathfrak{S}_4$, there is an element~$g$ of order~$9$ in $\operatorname{Aut}(S)$. By Lemma~\ref{l3} the eigenvalues of the representation~$g$ on the root lattice~$\mathrm{E}_6$ are
$$
\zeta_9,\ \zeta_9^2,\ \zeta_9^4,\ \zeta_9^5,\ \zeta_9^7,\ \zeta_9^8.
$$
By Lemma~\ref{l3} this means that the cube $g^3$ lies in the conjugacy class $A_2^3$. Therefore,~by Corollary~\ref{c3} there is a~nontrivial cube root of unity in $\mathbf{F}$.

The lemma is proved.
\end{proof}

\begin{lemma}
\label{l44}
Let $S$ be a~smooth cubic surface over a~field $\mathbf{F}$ of characteristic different from $3$. Assume that $\operatorname{Aut}(S) \supseteq (\mathbb{Z}/3\mathbb{Z})^3 \rtimes \mathfrak{S}_4$. Then $S$ is isomorphic to the Fermat cubic surface.
\end{lemma}

\begin{proof}
Since $\operatorname{Aut}(S) \supseteq (\mathbb{Z}/3\mathbb{Z})^3 \rtimes \mathfrak{S}_4$, by Theorem~\ref{t1} we have
$$
\operatorname{Aut}(S_{\mathbf{F}^{\mathrm{sep}}}) \simeq
\begin{cases}
 (\mathbb{Z}/3\mathbb{Z})^3 \rtimes \mathfrak{S}_4 &\text{if } \operatorname{char}\mathbf{F} \neq 2,
 \\
\mathrm{PSU}_4(\mathbb{F}_2) &\text{if }\operatorname{char}\mathbf{F} = 2,
\end{cases}
$$
and $S_{\mathbf{F}^{\mathrm{sep}}}$ is isomorphic to the Fermat cubic surface. By Lemma~\ref{l43} there is a~nontrivial cube root of unity $\omega \in \mathbf{F}$. Moreover, by Lemma~\ref{l42} the image $\Gamma$ of the absolute Galois group $\mathrm{Gal}(\mathbf{F}^{\mathrm{sep}}/\mathbf{F})$ in $W(\mathrm{E}_6)$ is trivial. Therefore, all lines on~$S_{\mathbf{F}^{\mathrm{sep}}}$ are defined over $\mathbf{F}$.

Consider the normal subgroup $G$ of $(\mathbb{Z}/3\mathbb{Z})^3 \rtimes \mathfrak{S}_4$ isomorphic to $(\mathbb{Z}/3\mathbb{Z})^3$. By Example~\ref{e1} and the uniqueness up to conjugation of the subgroup $(\mathbb{Z}/3\mathbb{Z})^3 \rtimes \mathfrak{S}_4$ in~$\mathrm{PSU}_4(\mathbb{F}_2)$, which follows from Lemma~\ref{l9}, up to conjugation over $\mathbf{F}^{\mathrm{sep}}$ the group~$G$ acts on the homogeneous coordinates $x$, $y$, $z$ and~$t$ on $\mathbb{P}^3$ by multiplication by cube roots of unity. By Lemma~\ref{l2} this realization holds up to conjugation over $\mathbf{F}$. Note that for any two monomials of degree~$3$ different from $x^3$, $y^3$, $z^3$ and~$t^3$ there exists~\mbox{$g \in (\mathbb{Z}/3\mathbb{Z})^3$} such that $g$ acts on these two monomials by multiplication by two different cube roots of unity. Therefore, the smooth cubic surfaces fixed by $G$ are defined by the equations
\begin{equation}
\label{eq6.2}
x^3+ay^3+bz^3+ct^3=0,
\end{equation}
where $a,b,c \in \mathbf{F}^*$.

Consider the hyperplane $T$ in $\mathbb{P}^3_{\mathbf{F}^{\mathrm{sep}}}$ defined by $x+\alpha y=0$, where $\alpha \in \mathbf{F}^{\mathrm{sep}}$ satisfies $\alpha^3=a$. There are three lines in the intersection of $T$ and $S_{\mathbf{F}^{\mathrm{sep}}}$.

Assume that $\alpha \notin \mathbf{F}$. Then there is an element $\tau \in \mathrm{Gal}(\mathbf{F}^{\mathrm{sep}}/\mathbf{F})$ satisfying
\begin{equation}
\label{eq6.3}
\tau(\alpha) \neq \alpha.
\end{equation}
The intersection of $\tau(T)$ and $S_{\mathbf{F}^{\mathrm{sep}}}$ also consists of three lines. However, the intersection of $T$ and $\tau(T)$ is a~line by~\eqref{eq6.3}. This means that not all lines on the intersection of $T$ and $S_{\mathbf{F}^{\mathrm{sep}}}$ are defined over $\mathbf{F}$. However, by Lemma~\ref{l42} the image~$\Gamma$ in $W(\mathrm{E}_6)$ of the absolute Galois group~$\mathrm{Gal}(\mathbf{F}^{\mathrm{sep}}/\mathbf{F})$ is trivial. Therefore, all lines on $S_{\mathbf{F}^{\mathrm{sep}}}$ are defined over $\mathbf{F}$. Thus, $\alpha \in \mathbf{F}$.

The same works for the hyperplanes $x+\beta z=0$ and $x+\gamma t=0$, where $\beta^3=b$ and~\mbox{$\gamma^3=c$,} respectively. Thus, $a,b,c \in \mathbf{F}^3$. Therefore, by the change of coordinates
$$
x \mapsto x,\qquad \sqrt[3]{a}\, y \to y,\qquad \sqrt[3]{b}\, z \to z \quad\text{and}\quad \sqrt[3]{c}\, t \to t
$$
the surface~\eqref{eq6.2} is transformed into
$$
x^3+y^3+z^3+t^3=0,
$$
which is the Fermat cubic surface.

The lemma is proved.
\end{proof}

\section{Fields without cube roots of unity}
\label{s7}

In this section we study smooth cubic surfaces over fields of characteristic different from~$3$ that do not contain nontrivial cube roots of unity. As in \S\,\ref{s6}, we regard~\mbox{$(\mathbb{Z}/3\mathbb{Z})^3 \rtimes \mathfrak{S}_4$} as a~group acting regularly 
 on the Fermat cubic surface over a~separably closed field of characteristic different from $2$ and $3$. This means that, up to conjugation,~$(\mathbb{Z}/3\mathbb{Z})^3$ acts on the homogeneous coordinates $x$, $y$, $z$ and~$t$ on~$\mathbb{P}^3$ by multiplication by cube roots of unity, $\mathfrak{S}_4$ acts by permutations of $x$, $y$, $z$ and~$t$. The first part of this section deals with possible automorphism groups of smooth cubic surfaces over fields without nontrivial cube roots of unity. Then for fields of characteristic~$5$ without nontrivial cube roots of unity we find the largest group among them and find smooth cubic surfaces with a~regular action of this group. In addition, we find all fields over which there is an isomorphism between the Fermat cubic surface and the surface~\eqref{eq1.2}.

\begin{lemma}
\label{l45}
Let $G$ be a~group isomorphic to a~subgroup of the automorphism group of the Fermat cubic surface over an algebraically closed field of characteristic different from $2$ and $3$. Consider the homomorphism
$$
\mathrm{pr} \colon (\mathbb{Z}/3\mathbb{Z})^3 \rtimes \mathfrak{S}_4 \to \mathfrak{S}_4
$$
and denote by $\mathrm{pr} |_G$ the restriction of $\mathrm{pr}$ to~$G$. Assume that $\mathrm{Ker}(\mathrm{pr} |_G)$ is neither trivial, nor isomorphic to $(\mathbb{Z}/3\mathbb{Z})^3$. Then $\mathrm{Coker}(\mathrm{pr} |_G)$ is isomorphic neither to $\mathfrak{S}_4$, nor to $\mathfrak{A}_4$.
\end{lemma}

\begin{proof}
Assume the converse. Then we have the exact sequence
$$
1 \to (\mathbb{Z}/3\mathbb{Z})^i \to G \xrightarrow{\mathrm{pr}|_{G}} \mathrm{Coker}(\mathrm{pr} |_G) \to 1,
$$
where $i=1$ or $2$ and $\mathrm{Coker}(\mathrm{pr}|_G)$ is either isomorphic to $\mathfrak{S}_4$, or to $\mathfrak{A}_4$. Then, as~$\mathrm{Ker}(\mathrm{pr}|_G)$ is an abelian group, we have the homomorphism
$$
f\colon \mathrm{Coker}(\mathrm{pr} |_G) \to \operatorname{Aut}( (\mathbb{Z}/3\mathbb{Z})^i ).
$$

\goodbreak

\noindent We have the isomorphisms
\begin{equation}
\label{eq7.1}
\operatorname{Aut}( \mathbb{Z}/3\mathbb{Z} ) \simeq \mathbb{Z}/2\mathbb{Z} \quad \text{and} \quad \operatorname{Aut}( (\mathbb{Z}/3\mathbb{Z})^2 ) \simeq \mathrm{GL}_2(\mathbb{F}_3).
\end{equation}

The group $\mathrm{GL}_2(\mathbb{F}_3)$ does not contain a~subgroup isomorphic to $\mathfrak{A}_4$, because otherwise the group $\mathrm{SL}_2(\mathbb{F}_3) \subset \mathrm{GL}_2(\mathbb{F}_3)$ would contain a~subgroup isomorphic to~$\mathfrak{A}_4$. However, this is false because $\mathrm{SL}_2(\mathbb{F}_3) \subset \mathrm{GL}_2(\mathbb{F}_3)$ does not contain a~subgroup isomorphic to the Klein four-group.

Thus, by~\eqref{eq7.1} this means that $\mathrm{Ker}(f)$ is isomorphic either
to $\mathfrak{S}_4$, to $\mathfrak{A}_4$, or to~$(\mathbb{Z}/2\mathbb{Z})^2$.
Therefore, either~$\mathfrak{S}_4$,~$\mathfrak{A}_4$,
or~$(\mathbb{Z}/2\mathbb{Z})^2$ acts trivially by conjugations
on~$(\mathbb{Z}/3\mathbb{Z})^i$. By Lemma~\ref{l18} the centralizer of any
nonidentity element of the normal subgroup~$ (\mathbb{Z}/3\mathbb{Z})^3$ of
$(\mathbb{Z}/3\mathbb{Z})^3 \rtimes \mathfrak{S}_4$ is isomorphic
to~$(\mathbb{Z}/3\mathbb{Z})^3 \rtimes D$, where $D$ is a~subgroup of
$\mathfrak{S}_4$ isomorphic to $\mathfrak{S}_3$,
$(\mathbb{Z}/2\mathbb{Z})^2$, or $\mathbb{Z}/2\mathbb{Z}$. Neither
$\mathfrak{S}_4$, nor~$\mathfrak{A}_4$ lies in these~groups. The only
possibility is when $\mathrm{Ker}(f)\mkern-2mu \simeq \mkern-1mu (\mathbb{Z}/2\mathbb{Z})^2$ and
${\mathrm{Coker}(\mathrm{pr} |_G) \mkern-2mu\simeq \mathfrak{A}_4}$. In this case $i=2$.
One can see that there are only $6$ elements of order~$3$ in
$\mathrm{Ker}(\mathrm{pr})$ such that the group $D$ is isomorphic to
$(\mathbb{Z}/3\mathbb{Z})^2$ by the proof of Lemma~\ref{l18}. This means that
this case is impossible. This contradiction completes the proof.
\end{proof}

\begin{lemma}
\label{l46}
Let $S$ be a~smooth cubic surface over a~field $\mathbf{F}$ of characteristic different from $2$ and~$3$. Assume that $\mathbf{F}$ does not contain a~nontrivial cube root of unity. Then $|{\operatorname{Aut}(S)}| \leqslant 120$.
\end{lemma}

\begin{proof}
Assume that
$$
|{\operatorname{Aut}(S)}| > 120.
$$
Then by Theorem~\ref{t1}, parts $(\mathrm{iii})'$, $(\mathrm{iii})''$, $(\mathrm{iv})'$ and $(\mathrm{iv})''$, we have
$$
\operatorname{Aut}(S) \subseteq (\mathbb{Z}/3\mathbb{Z})^3 \rtimes \mathfrak{S}_4.
$$
Consider the homomorphism
$$
\mathrm{pr} \colon (\mathbb{Z}/3\mathbb{Z})^3 \rtimes \mathfrak{S}_4 \to \mathfrak{S}_4
$$
which is the natural projection onto $\mathfrak{S}_4$. Consider the restriction $\mathrm{pr} |_{\operatorname{Aut}(S)}$ of $\mathrm{pr}$ to the subgroup $\operatorname{Aut}(S)$. The kernel of $\mathrm{pr}|_{\operatorname{Aut}(S)}$ is a~subgroup of the normal subgroup~$(\mathbb{Z}/3\mathbb{Z})^3$ of $(\mathbb{Z}/3\mathbb{Z})^3 \rtimes \mathfrak{S}_4$. Since $\mathbf{F}$ does not contain nontrivial cube roots of unity, from Corollary~\ref{c3} and Lemma~\ref{l17} we get that
\begin{equation}
\label{eq7.2}
\mathrm{Ker}(\mathrm{pr} |_{\operatorname{Aut}(S)}) \subsetneq (\mathbb{Z}/3\mathbb{Z})^3.
\end{equation}
Since $648 \geqslant |{\operatorname{Aut}(S)}| > 120$, the group $\operatorname{Aut}(S)$ has order~$162$, $216$, $324$, or $648$. So,~by~\eqref{eq7.2} the only possibility for the order of $\operatorname{Aut}(S)$ is~$216$, that is, for $\operatorname{Aut}(S)$ the following exact sequence
$$
1 \to (\mathbb{Z}/3\mathbb{Z})^2 \to \operatorname{Aut}(S) \xrightarrow{\mathrm{pr}|_{\operatorname{Aut}(S)}} \mathfrak{S}_4 \to 1
$$
holds. By Lemma~\ref{l45} this is impossible. This contradiction completes the proof of the inequality
$$
|{\operatorname{Aut}(S)}| \leqslant 120.
$$

The lemma is proved.
\end{proof}

\begin{lemma}
\label{l47}
Let $S$ be a~smooth cubic surface over a~field $\mathbf{F}$ of characteristic~$2$. Assume that $\mathbf{F}$ contains no nontrivial cube roots of unity. Then
$$
|{\operatorname{Aut}(S)}| \leqslant 720.
$$
\end{lemma}

\begin{proof}
Assume that $|{\operatorname{Aut}(S)}| > 720$. Then by Theorem~\ref{t1},~$(\mathrm{i})'$ and $(\mathrm{i})''$, we obtain an isomorphism $\operatorname{Aut}(S_{\mathbf{F}^{\mathrm{sep}}}) \simeq \mathrm{PSU}_4(\mathbb{F}_2)$. From Lemma~\ref{l9}, using the restriction on the order of the automorphism group we obtain
$$
\operatorname{Aut}(S) \supseteq (\mathbb{Z}/2\mathbb{Z})^4 \rtimes \mathfrak{A}_5.
$$

 By Lemma~\ref{l11} this means that the image of the absolute Galois group $\mathrm{Gal}(\mathbf{F}^{\mathrm{sep}}/\mathbf{F})$ in the Weyl group $W(\mathrm{E}_6)$ is trivial. In other words, any line on $S_{\mathbf{F}^{\mathrm{sep}}}$ is defined over~$\mathbf{F}$.

Consider the element $g \in \operatorname{Aut}(S)$ of order~$5$. By Lemma~\ref{l12},\,(i), there is a~$g$-invari\-ant line $l \subset S$. The action of $g$ on $l$ is nontrivial since otherwise any line intersecting~$l$ would be $g$-invariant, which is impossible by Lemma~\ref{l12},\,(i), and because there are at most three lines meeting in a~single point on a~smooth cubic surface (see Lemma~\ref{l13}). This means that
\begin{equation}
\label{eq7.3}
\mathbb{Z}/5\mathbb{Z} \subset \mathrm{PGL}_2(\mathbf{F}).
\end{equation}
However, as $\mathbf{F}$ does not contain nontrivial cube roots of unity, the inclusion~\eqref{eq7.3} is impossible by Lemma~\ref{l10}. This contradiction completes the proof.
\end{proof}

\begin{lemma}
\label{l48}
Let $S$ be a~smooth cubic surface over a~field $\mathbf{F}$ of characteristic different from $2$, $3$ and~$5$ such that $\mathbf{F}$ does not contain nontrivial cube roots of unity. Then
$$
|{\operatorname{Aut}(S)}| \leqslant 120.
$$
Moreover, equality holds if and only if $\operatorname{Aut}(S) \simeq \mathfrak{S}_5$ and $S$ is isomorphic to the Clebsch cubic surface.
\end{lemma}

\begin{proof}
By Lemma~\ref{l46} the inequality $|{\operatorname{Aut}(S)}| \leqslant 120$ holds. By Theorem~\ref{t1}, $(\mathrm{iv})'$ and~$(\mathrm{iv})''$, equality holds if and only if
$$
\operatorname{Aut}(S) \simeq \mathfrak{S}_5.
$$
By Example~\ref{e2} the group~$\mathfrak{S}_5$ is realized over $\mathbf{F}$ as the automorphism group of the Clebsch cubic surface.

Let $\operatorname{Aut}(S) \simeq \mathfrak{S}_5$. Then by Corollary~\ref{c5} the surface $S$ is isomorphic to the Clebsch cubic surface.

The lemma is proved.
\end{proof}

\begin{lemma}
\label{l49}
Let $S$ be a~smooth cubic surface over a~field $\mathbf{F}$ of characteristic~$2$ such that $\mathbf{F}$ does not contain nontrivial cube roots of unity. Then
$$
|{\operatorname{Aut}(S)}| \leqslant 720.
$$
Moreover, equality holds if and only if $\operatorname{Aut}(S) \simeq \mathfrak{S}_6$ and $S$ is isomorphic to~\eqref{eq1.2}.
\end{lemma}

\begin{proof}
By Lemma~\ref{l47} we have $|{\operatorname{Aut}(S)}| \leqslant 720$. If equality holds, then by Theorem~\ref{t1},~$(\mathrm{i})'$ and $(\mathrm{i})''$, we have
$$
\operatorname{Aut}(S) \subset \mathrm{PSU}_4(\mathbb{F}_2).
$$
By Lemma~\ref{l9} this means that $\operatorname{Aut}(S) \simeq \mathfrak{S}_6$. By Theorem~\ref{t2},\,(ii), this group is realized as the automorphism group of smooth cubic surface~\eqref{eq1.2}. Using Corollaries~\ref{c5} and~\ref{c6},\,(ii), we get that $S$ is isomorphic to cubic surface~\eqref{eq1.2}.

The lemma is proved.
\end{proof}

Before we consider the automorphism groups of smooth cubic surfaces over a~field of characteristic~$5$ not containing nontrivial cube roots of unity, we present some auxiliary lemmas relating to subgroups of the group $(\mathbb{Z}/3\mathbb{Z})^3 \rtimes \mathfrak{S}_4$.

\begin{lemma}
\label{l50}
Up to conjugation the group $(\mathbb{Z}/3\mathbb{Z})^3 \rtimes \mathfrak{S}_4$ contains a~unique subgroup isomorphic to $(\mathbb{Z}/3\mathbb{Z})^2 \rtimes \mathfrak{D}_4$X.
\end{lemma}

\begin{proof}
Let $G \simeq (\mathbb{Z}/3\mathbb{Z})^2 \rtimes_{\phi} \mathfrak{D}_4$ be a~subgroup of $(\mathbb{Z}/3\mathbb{Z})^3 \rtimes \mathfrak{S}_4$ for some homomorphism
$$
\phi\colon \mathfrak{D}_4 \to \operatorname{Aut}( (\mathbb{Z}/3\mathbb{Z})^2 ) \simeq \mathrm{GL}_2(\mathbb{F}_3).
$$

We show that all such subgroups of $(\mathbb{Z}/3\mathbb{Z})^3 \rtimes \mathfrak{S}_4$ are conjugate. Consider the projection homomorphism
$$
\mathrm{pr} \colon (\mathbb{Z}/3\mathbb{Z})^3 \rtimes \mathfrak{S}_4 \to \mathfrak{S}_4.
$$

Then we have $\mathrm{pr}(G) \simeq \mathfrak{D}_4$. In the group $(\mathbb{Z}/3\mathbb{Z})^3 \rtimes \mathfrak{S}_4$ we can view the action of $\mathfrak{S}_4$ on~$(\mathbb{Z}/3\mathbb{Z})^3$ as a~representation
$$
\rho \colon \mathfrak{S}_4 \to \mathrm{GL}(V),
$$
where $V$ is the vector space of dimension $3$ over $\mathbb{F}_3$ corresponding to the normal subgroup $ (\mathbb{Z}/3\mathbb{Z})^3$ of $(\mathbb{Z}/3\mathbb{Z})^3 \rtimes \mathfrak{S}_4$. The homomorphism $\rho$ is injective, so $\rho |_{\mathfrak{D}_4}$ is also injective.

Let us prove that the action of~$\mathfrak{D}_4$ on $V$ preserves a~subspace of dimension~$2$. Assume the converse. Then either $V$ is irreducible, or there is an invariant subspace~$L$ of dimension~$1$ in~$V$. The first case is impossible, because otherwise it would mean that there is an irreducible representation of dimension~$3$ of the group~$\mathfrak{D}_4$ over~$\overline{\mathbb{F}}_3$, which is impossible since the characteristic of the field cannot divide~the dimension of a~representation by Corollary~4.6 in~\cite{12}, Ch.~18. Assume that the~second case holds. Then by Maschke's theorem (see, for instance, Theorem~1.2 in~\cite{12}, Ch.~18) the space~$V$ is a~direct sum of subrepresentations of dimension~$1$. This would mean that $\mathfrak{D}_4$ is commutative.

So by Maschke's theorem $V$ is a~direct sum of subrepresentations of dimensions~$1$ and $2$. This shows the existence and uniqueness of a~subspace of dimension $2$ of $V$ that is invariant under the action of $\mathfrak{D}_4$. Since $\mathfrak{D}_4$ is a~unique subgroup of order~$8$ up to conjugation in $\mathfrak{S}_4$, we obtain the required result.

The lemma is proved.
\end{proof}

\begin{remark}
\label{r6}
From now on we mean by $(\mathbb{Z}/3\mathbb{Z})^2 \rtimes \mathfrak{D}_4$ a~subgroup of $(\mathbb{Z}/3\mathbb{Z})^3 \rtimes \mathfrak{S}_4$, which is unique up to conjugation by Lemma~\ref{l50}.
\end{remark}

\begin{example}
\label{e4}
Let $\mathbf{F}$ be a~field not containing nontrivial cube roots of unity. Consider the projective space $\mathbb{P}^3$ with homogeneous coordinates $x$, $y$, $z$ and~$t$. Consider the subgroup isomorphic to $(\mathbb{Z}/3\mathbb{Z})^2$ in the automorphism group of $\mathbb{P}^3$ that is generated by
\begin{equation}
\label{eq7.4}
\begin{pmatrix}
-1 & 1 & 0 & 0\\
-1 & 0 & 0 & 0\\
0 & 0 & 1 & 0\\
0 & 0 & 0 & 1
\end{pmatrix} \quad \text{and} \quad
\begin{pmatrix}
1 & 0 & 0 & 0\\
0 & 1 & 0 & 0\\
0 & 0 & -1 & 1\\
0 & 0 & -1 & 0
\end{pmatrix}.
\end{equation}
Consider the group $\mathfrak{D}_4 \subset \mathrm{PGL}_4(\mathbf{F})$ generated by the permutations~$(12)$ and~$(1324)$ of the homogeneous coordinates $x$, $y$, $z$ and~$t$ in~$\mathbb{P}^3$. Then $(\mathbb{Z}/3\mathbb{Z})^2$ and~$\mathfrak{D}_4$ generate the group $(\mathbb{Z}/3\mathbb{Z})^2 \rtimes \mathfrak{D}_4$ mentioned in Remark~\ref{r6}.
\end{example}

\begin{lemma}
\label{l51}
Let $G$ be a~subgroup of the group $ (\mathbb{Z}/3\mathbb{Z})^3 \rtimes \mathfrak{S}_4 \subset W(\mathrm{E}_6)$. Assume that $G$ does not contain elements in the conjugacy class $A^3_2$. Then $|G| \leqslant 72$, and equality holds if and only if
$$
G \simeq (\mathbb{Z}/3\mathbb{Z})^2 \rtimes \mathfrak{D}_4.
$$
Moreover, in the case of equality the group $G$ with such properties is unique up to conjugation.
\end{lemma}

\begin{proof}
Consider the projection homomorphism
$$
\mathrm{pr} \colon (\mathbb{Z}/3\mathbb{Z})^3 \rtimes \mathfrak{S}_4 \to \mathfrak{S}_4.
$$

We denote by $\mathrm{pr}|_{G}$ the restriction of the homomorphism $\mathrm{pr}$ to the group $G$. By Lemma~\ref{l17} the kernel $\mathrm{Ker}(\mathrm{pr}|_{G})$ lies in $(\mathbb{Z}/3\mathbb{Z})^2$. So by Lemma~\ref{l45} either $\mathrm{pr}|_{G}$ is injective, or the order of $\mathrm{pr}(G)$ is at most~$8$. This means that $|G| \leqslant 72$.

Consider the case when $|G|=72$. We have two possibilities for $G$:
\begin{equation}
1 \to \mathbb{Z}/3\mathbb{Z} \to G \xrightarrow{\mathrm{pr}|_{G}} \mathfrak{S}_4 \to 1,
\label{eq7.5}
\end{equation}
or
\begin{equation}
1 \to (\mathbb{Z}/3\mathbb{Z})^2 \to G \xrightarrow{\mathrm{pr}|_{G}} \mathfrak{D}_4 \to 1. \label{eq7.6}
\end{equation}

By Lemma~\ref{l45} the case~\eqref{eq7.5} is not realized. Consider the case~\eqref{eq7.6}. By the Schur--Zassenhaus theorem (see~\cite{9}, \S\,17.4, Theorem~39) $$
G \simeq (\mathbb{Z}/3\mathbb{Z})^2 \rtimes_{\phi} \mathfrak{D}_4
$$
for some homomorphism 
$$
\phi\colon \mathfrak{D}_4 \to \operatorname{Aut}( (\mathbb{Z}/3\mathbb{Z})^2 ) \simeq \mathrm{GL}_2(\mathbb{F}_3).
$$
By Lemma~\ref{l50} such a~group 
exists and is unique up to conjugation.

The lemma is proved.
\end{proof}

\begin{lemma}
\label{l52}
Let $\mathbf{F}$ be a~field of characteristic~$5$ not containing nontrivial cube roots of unity. Let $S$ be a~smooth cubic surface over $\mathbf{F}$ such that
$$
\operatorname{Aut}(S_{\mathbf{F}^{\mathrm{sep}}}) \simeq \mathcal{H}_3(\mathbb{F}_3) \rtimes \mathbb{Z}/4\mathbb{Z}.
$$
Then $\operatorname{Aut}(S) \subsetneq \mathcal{H}_3(\mathbb{F}_3) \rtimes \mathbb{Z}/4\mathbb{Z}$.
\end{lemma}

\begin{proof}
By Corollary~\ref{c2} no elements of $\operatorname{Aut}(S)$ belong to the conjugacy class~$A^3_2$. But by Table~8 in~\cite{8} and Remark~\ref{r1} the group $\operatorname{Aut}(S_{\mathbf{F}^{\mathrm{sep}}})$ contains elements lying in the conjugacy class $A^3_2$. Therefore, we have the strict inclusion
$$
\operatorname{Aut}(S) \subsetneq \mathcal{H}_3(\mathbb{F}_3) \rtimes \mathbb{Z}/4\mathbb{Z}.
$$

The lemma is proved.
\end{proof}

\begin{corollary}
\label{c7}
Let $\mathbf{F}$ be a~field of characteristic~$5$ not containing nontrivial cube roots of unity. Let $S$ be a~smooth cubic surface over $\mathbf{F}$. Then $|{\operatorname{Aut}(S)}| \leqslant 72$. Moreover, equality holds if and only if
$$
\operatorname{Aut}(S) \simeq (\mathbb{Z}/3\mathbb{Z})^2 \rtimes \mathfrak{D}_4.
$$
\end{corollary}

\begin{proof}
By Lemma~\ref{l46} \ $|{\operatorname{Aut}(S)}| \leqslant 120$. By Theorem~\ref{t1},\,$(\mathrm{iii})''$, this implies the inequality $|{\operatorname{Aut}(S)}| \leqslant 108$. Assume that
$$
\operatorname{Aut}(S) \subset (\mathbb{Z}/3\mathbb{Z})^3 \rtimes \mathfrak{S}_4.
$$
Then by Lemma~\ref{l51} we have $|{\operatorname{Aut}(S)}| \leqslant 72$, and equality holds if and only if
$$
\operatorname{Aut}(S) \simeq (\mathbb{Z}/3\mathbb{Z})^2 \rtimes \mathfrak{D}_4.
$$

Otherwise, assume that $\operatorname{Aut}(S)$ lies in the
 second
largest automorphism group of a~smooth cubic surface over an algebraically closed field of characteristic~$5$. By~Theorem~\ref{t1},\,$(\mathrm{iii})''$, this means that
$$
\operatorname{Aut}(S) \subset \mathcal{H}_3(\mathbb{F}_3) \rtimes \mathbb{Z}/4\mathbb{Z}.
$$
However, by Lemma~\ref{l52} this means that the order of $\operatorname{Aut}(S)$ is at most $54$, which is less than~$72$. By Theorem~\ref{t1},\,$(\mathrm{iii})''$, the third largest automorphism group of a~smooth cubic surface over an algebraically closed field of characteristic~$5$ is the group $\mathcal{H}_3(\mathbb{F}_3) \rtimes \mathbb{Z}/2\mathbb{Z}$. Its order is $54$ and, again, it is less than $72$.

The corollary is proved.
\end{proof}

Now we are ready to study smooth cubic surfaces with a~regular action of the group $(\mathbb{Z}/3\mathbb{Z})^2 \rtimes \mathfrak{D}_4$.

\begin{lemma}
\label{l53}
Let $\mathbf{F}$ be a~field of characteristic different from $2$ and~$3$ such that it contains no nontrivial cube roots of unity. Let $S$ be a~smooth cubic surface over~$\mathbf{F}$ such that its automorphism group is isomorphic to $(\mathbb{Z}/3\mathbb{Z})^2 \rtimes \mathfrak{D}_4$. Then $S$ is isomorphic to the cubic surface
$$
2(x^3+y^3+z^3+t^3)-3(x^2y+xy^2+z^2t+zt^2)=0.
$$
\end{lemma}

\begin{proof}
First of all, recall that by Lemma~\ref{l50} and Remark~\ref{r6} the group
$$
(\mathbb{Z}/3\mathbb{Z})^2 \rtimes \mathfrak{D}_4
$$
is isomorphic to the group defined in Example~\ref{e4}.

Consider all cubic surfaces invariant under the action of $(\mathbb{Z}/3\mathbb{Z})^2$ generated by the elements~\eqref{eq7.4}. These cubic surfaces are defined by homogeneous polynomials~$f(x,y,z,t)$ of degree~$3$ such that each element $g \in (\mathbb{Z}/3\mathbb{Z})^2$ preserves $f(x,y,z,t)$ up to a~scalar factor. Since the matrices~\eqref{eq7.4} are of order~$3$ and $\mathbf{F}$ does not contain nontrivial cube roots of unity, this scalar is equal to~$1$. Thus, first let us find all homogeneous polynomials of degree~$3$ invariant under the action of
\begin{equation}
\label{eq7.7}
\begin{pmatrix}
-1 & 1 & 0 & 0\\
-1 & 0 & 0 & 0\\
0 & 0 & 1 & 0\\
0 & 0 & 0 & 1
\end{pmatrix}.
\end{equation}

Consider a~homogeneous polynomial
\begin{align}
\notag
f(x,y,z,t)
&=ax^3+by^3+cx^2y+dxy^2
\\
&\qquad+x^2l_1(z,t)+y^2l_2(z,t)+xyl_3(z,t)+xg_1(z,t)+yg_2(z,t)+h(z,t),
\label{eq7.8}
\end{align}
where $a,b,c,d \in \mathbf{F}$, $l_1(z,t)$, $l_2(z,t)$ and $l_3(z,t)$ are linear homogeneous polynomials,~$g_1(z,t)$ and~$g_2(z,t)$ are quadratic homogeneous polynomials, and $h(z,t)$ is a~cubic homogeneous polynomial. The element~\eqref{eq7.7} maps $f(x,y,z,t)$ to
\begin{align*}
\widetilde{f}(x,y,z,t)
&=(-a-b-c-d)x^3+ay^3+(3a+2c+d)x^2y-(3a+c)xy^2
\\
&\qquad
+(l_1(z,t)+l_2(z,t)+l_3(z,t))x^2+l_1(z,t)y^2{-}\,(2l_1(z,t)+l_3(z,t))xy
\\
&\qquad
-(g_1(z,t)+g_2(z,t))x+g_1(z,t)y+h(z,t).
\end{align*}
It is not difficult to see that a~homogeneous polynomial $f(x,y,z,t)$ such that
$$
f(x,y,z,t)=\widetilde{f}(x,y,z,t)
$$
is of the form
$$
f(x,y,z,t)=a(x^3+y^3-3xy^2)+bxy(x-y)+l(z,t)(x^2+y^2-xy)+h(z,t),
$$
where $a,b \in \mathbf{F}$, $l(z,t)$ is a~linear homogeneous polynomial, and $h(z,t)$ is a~cubic homogeneous polynomial. Applying the same reasoning to the matrix
$$
\begin{pmatrix}
1 & 0 & 0 & 0\\
0 & 1 & 0 & 0\\
0 & 0 & -1 & 1\\
0 & 0 & -1 & 0
\end{pmatrix}
$$
gives us a~homogeneous polynomial
\begin{equation}
\label{eq7.9}
f(x,y,z,t)=a(x^3+y^3-3xy^2)+bxy(x-y)+c(z^3+t^3-3zt^2)+dzt(z-t),
\end{equation}
where $a,b,c,d \in \mathbf{F}$. So all cubic surfaces invariant under the action of~$(\mathbb{Z}/3\mathbb{Z})^2$ generated by the elements~\eqref{eq7.4} are of the form~\eqref{eq7.9}. Among the surface defined by polynomials~\eqref{eq7.9} let us find the cubic surfaces that are also semi-invariant under the action of the group $\mathfrak{D}_4$ generated by the permutations $(12)$ and $(1324)$ of~the coordinates $x$, $y$, $z$ and~$t$. Consider the permutation $(12)$. It takes~\eqref{eq7.9} to the~homogeneous polynomial
\begin{equation}
\label{eq7.10}
\widehat{f}(x,y,z,t)=a(x^3+y^3-3yx^2)-bxy(x-y)+c(z^3+t^3-3zt^2)+dzt(z-t).
\end{equation}
We can see from~\eqref{eq7.10} that if $f(x,y,z,t)=-\widehat{f}(x,y,z,t)$, then $f(x,y,z,t)=0$ defines a~singular cubic surface. This means that the polynomials~\eqref{eq7.9} invariant under the action of the permutation $(12)$ and defining smooth cubic surfaces are of the form
\begin{equation}
\label{eq7.11}
f(x,y,z,t)=a(2(x^3+y^3)-3(x^2y+xy^2))+c(z^3+t^3-3zt^2)+dzt(z-t),
\end{equation}
where $a \in \mathbf{F}^*$ and $c,d \in \mathbf{F}$ are such that~\eqref{eq7.11} defines a~smooth surface. The same reasoning, as applied to the permutation
$$
(12)(1324)^2=(34),
$$
shows that smooth cubic surfaces must be defined by polynomials of the form
\begin{equation}
\label{eq7.12}
f(x,y,z,t)=a(2(x^3+y^3)-3(x^2y+xy^2))+c(2(z^3+t^3)-3(z^2t+zt^2)),
\end{equation}
where $a,c \in \mathbf{F}^*$. Using the semi-invariance of~\eqref{eq7.12} under the permutation $(1324)$ we obtain
\begin{equation}
\label{eq7.13}
f(x,y,z,t)=2(x^3+y^3)- 3(x^2y+xy^2)\pm (2(z^3+t^3)-3(z^2t+zt^2)).
\end{equation}
However, the change of coordinates
$$
x \mapsto x, \qquad y \mapsto y, \qquad z \mapsto -z \quad\text{and}\quad t \mapsto -t
$$
shows that the smooth cubic surfaces defined by the polynomials~\eqref{eq7.13} with plus and minus signs are isomorphic.

Lemma~\ref{l53} is proved.
\end{proof}

\begin{lemma}
\label{l54}
Let $\mathbf{F}$ be a~field of characteristic~$5$ not containing nontrivial cube roots of unity. Let $S$ be a~smooth cubic surface over $\mathbf{F}$ such that its automorphism group is isomorphic to $(\mathbb{Z}/3\mathbb{Z})^2 \rtimes \mathfrak{D}_4$. Then $S$ is isomorphic to the cubic surface~\eqref{eq1.2}.
\end{lemma}

\begin{proof}
From Lemma~\ref{l53} we get that $S$ is isomorphic to
\begin{equation}
\label{eq7.14}
2(x^3+y^3+z^3+t^3)-3(x^2y+xy^2+z^2t+zt^2)=0.
\end{equation}
Since $2=-3$ in the field $\mathbf{F}$, the cubic surface~\eqref{eq7.14} is defined by the equation
\begin{equation}
\label{eq7.15}
x^3+y^3+z^3+t^3+x^2y+xy^2+z^2t+zt^2=0.
\end{equation}

By the change of coordinates
$$
x \mapsto x+ 2y, \qquad y \mapsto 2x+ y, \qquad z \mapsto z+ 2t \quad\text{and}\quad t \mapsto 2z+ t
$$
the equation~\eqref{eq7.15} is transformed into
\begin{equation}
\label{eq7.16}
x^2y+xy^2+z^2t+zt^2=0.
\end{equation}
By the change of coordinates $[x:y:z:t] \mapsto [x:t:z:y]$ the equation~\eqref{eq7.16} is transformed into
$$
 x^2t+y^2z+z^2y+t^2x=0.
$$

The lemma is proved.
\end{proof}

Now we prove that the Fermat cubic surface is not isomorphic to the surface~\eqref{eq1.2} over fields without cube roots of unity.

\begin{lemma}
\label{l55}
Let $\mathbf{F}$ be a~field of characteristic different from $3$ such that it does not contain nontrivial cube roots of unity. Then the Fermat cubic surface is not isomorphic to the surface~\eqref{eq1.2}.
\end{lemma}

\begin{proof}
Consider these cubic surfaces over $\mathbf{F}^{\mathrm{sep}}$. The lines on the Fermat cubic surface over $\mathbf{F}^{\mathrm{sep}}$ are defined by the systems of equations in $\mathbb{P}^3_{\mathbf{F}^{\mathrm{sep}}}$
\begin{equation}
\label{eq7.17}
\begin{cases}
x+\omega^i y=0,\\
z+\omega^jt=0,
\end{cases}
\qquad
\begin{cases}
x+\omega^i z=0,\\
t+\omega^jy=0
\end{cases}
\quad\text{and}\quad
\begin{cases}
x+\omega^i t=0,\\
z+\omega^jy=0,
\end{cases}
\end{equation}
where $\omega \in \mathbf{F}^{\mathrm{sep}}$ is a~nontrivial cube root of unity and $i,j \in \{0,1,2\}$. Let $\mathbf{L}$ be the extension of degree~$2$ of $\mathbf{F}$ such that $\mathbf{L}=\mathbf{F}[t]/(t^2+t+1)$. Then we immediately see from~\eqref{eq7.17} that all lines on the Fermat cubic surface over $\mathbf{F}^{\mathrm{sep}}$ are defined over~$\mathbf{L}$. Since the nonidentity element of the Galois group
$$
\mathrm{Gal}(\mathbf{L}/\mathbf{F}) \simeq \mathbb{Z}/2\mathbb{Z}
$$
maps $\omega$ to~$\omega^2$, there are only three lines on the Fermat surface that are defined over~$\mathbf{F}$, namely,
$$
\begin{cases}
x+ y=0,\\
z+t=0,
\end{cases}
\qquad
\begin{cases}
x+ z=0,\\
t+y=0
\end{cases}
\quad\text{and}\quad
\begin{cases}
x+ t=0,\\
z+y=0.
\end{cases}
$$

However, more than three lines on the surface~\eqref{eq1.2} are defined over $\mathbf{F}$, for instance,
$$
\begin{cases}
x=0,\\
z=0,
\end{cases}
\qquad
\begin{cases}
y=0,\\
z=0,
\end{cases}
\qquad
\begin{cases}
x=0,\\
t=0,
\end{cases}
\qquad
\begin{cases}
x=0,\\
z+t=0
\end{cases} \text{and so on}.
$$
This means that these cubic surfaces are not isomorphic.

The lemma is proved.
\end{proof}

After this lemma one wonders if the Fermat cubic surface is isomorphic to the surface~\eqref{eq1.2} over fields containing nontrivial cube roots of unity. The following lemma gives a~positive answer to this question.

\begin{lemma}
\label{l56}
Let $\mathbf{F}$ be a~field of characteristic different from~$3$ that contains a~nontrivial cube root of unity. Then the Fermat cubic surface is isomorphic to the surface~\eqref{eq1.2}.
\end{lemma}

\begin{proof}
Let $\omega \in \mathbf{F}$ be a~nontrivial cube root of unity. Then the isomorphism between the Fermat cubic surface and the surface~\eqref{eq1.2} is given by the collineation
$$
[x:y:z:t] \mapsto [x+y:\omega x+\omega^2 y:z+t:\omega z+\omega^2 t],
$$
which maps the surface~\eqref{eq1.2} to the Fermat cubic surface.

The lemma is proved.
\end{proof}

\section{Proof of the main results}
\label{s8}

In this section we prove Theorem~\ref{t3} and Proposition~\ref{p1}.

\begin{proof}[of Theorem~\ref{t3}]
Let us prove~$(\mathrm{i})_2$. The inequality $|{\operatorname{Aut}(S)}| \leqslant 25\,920$ follows from Theorem~\ref{t1},\,$(\mathrm{i})'$. The equality case follows from~\cite{8}, \S\,5.1. The subgroup of order~$25\,920$ in $W(\mathrm{E}_6)$ is isomorphic to $\mathrm{PSU}_4(\mathbb{F}_2)$. By Lemma~\ref{l9} it contains a~subgroup isomorphic to  $(\mathbb{Z}/3\mathbb{Z})^3 \rtimes \mathfrak{S}_4$. Hence by Lemma~\ref{l44} the smooth cubic surface~$S$ with a~regular action of  $(\mathbb{Z}/3\mathbb{Z})^3 \rtimes \mathfrak{S}_4$  is unique up to isomorphism and isomorphic to the Fermat cubic surface.

Assertion~$(\mathrm{ii})_2$ follows from Lemma~\ref{l49}.

Let us prove $(\mathrm{i})_3$. The inequality $|{\operatorname{Aut}(S)}| \leqslant 216$ follows from Theorem~\ref{t1},\,$(\mathrm{ii})'$. The equality condition holds by Lemma~\ref{l28}. The group of order~$216$ that is the automorphism group of a~smooth cubic surface is also unique by Theorem~\ref{t1},\,$(\mathrm{ii})'$, and isomorphic to $ \mathcal{H}_3(\mathbb{F}_3) \rtimes \mathbb{Z}/8\mathbb{Z}$. By Lemma~\ref{l36} the smooth cubic surface $S$
such that
$$
\operatorname{Aut}(S) \simeq \mathcal{H}_3(\mathbb{F}_3) \rtimes \mathbb{Z}/8\mathbb{Z}
$$
is unique up to isomorphism and isomorphic to the smooth surface~\eqref{eq1.1}.

Assertion $(\mathrm{ii})_3$ follows directly from Lemma~\ref{l40}.

Let us prove assertions $(\mathrm{i})_5$ and~$(\mathrm{i})_{\neq 2,3,5}$. The inequality $|{\operatorname{Aut}(S)}| \leqslant 648$ follows from Theorem~\ref{t1}, $(\mathrm{iii})'$ and $(\mathrm{iv})'$. The equality condition holds by Example~\ref{e1}. The group of order~$648$ that is the automorphism group of a~smooth cubic surface is unique by Theorem~\ref{t1}, $(\mathrm{iii})'$ and~$(\mathrm{iv})'$, and isomorphic to $(\mathbb{Z}/3\mathbb{Z})^3 \rtimes \mathfrak{S}_4$. So by Lemma~\ref{l44} the smooth cubic surface $S$ with such an automorphism group is unique up to isomorphism and isomorphic to the Fermat cubic.

Now we prove~$(\mathrm{ii})_5$. By Corollary~\ref{c7} \ $|{\operatorname{Aut}(S)}| \leqslant 72$ and equality holds if and only if $\operatorname{Aut}(S)$ is isomorphic to $(\mathbb{Z}/3\mathbb{Z})^2 \rtimes \mathfrak{D}_4$. By Lemma~\ref{l54} the smooth cubic surface~$S$ is unique up to isomorphism and isomorphic to the surface~\eqref{eq1.2}.

Assertion~$(\mathrm{ii})_{\neq 2,3,5}$ follows from Lemma~\ref{l48}.

Theorem~\ref{t3} is proved.
\end{proof}

Before we prove Proposition~\ref{p1}, we state the following lemmas.

\begin{lemma}
\label{l57}
Let $\mathbf{F}$ be a~separably closed field of characteristic different from~$2$ and~$3$. Then the automorphism group of the cubic surface~\eqref{eq1.1} over $\mathbf{F}$ is isomorphic to~$\mathbb{Z}/8\mathbb{Z}$.
\end{lemma}

For characteristic zero this result follows from Theorem~9.5.8 in~\cite{7}. For a~positive characteristic it follows from Lemma~12.12 in~\cite{8} by applying Theorem~\ref{t4},\,(i).

\begin{lemma}
\label{l58} The
cubic surface~\eqref{eq1.1} is isomorphic neither to the Fermat cubic surface, nor to the Clebsch one, nor to the surface~\eqref{eq1.2}.
\end{lemma}

\begin{proof}
Over a~field of characteristic~$2$ the surface~\eqref{eq1.1} has a~singularity at the point~\mbox{$[0:0:1:1]$,} while the other surfaces listed above are smooth.

Over a~field of characteristic~$3$ the Fermat cubic surface is nonreduced and the surface~\eqref{eq1.2} is singular, while both the Clebsch cubic surface and the surface~\eqref{eq1.1} are smooth. However, the last two cubic surfaces are not isomorphic, since by Theorem~\ref{t1},~$(\mathrm{ii})'$ and $(\mathrm{ii})''$, they have different automorphism groups over an algebraically closed field of characteristic~$3$.

Let $\mathbf{F}$ be a~field of characteristic different from $2$ and~$3$. By Lemma~\ref{l57} the automorphism group of the cubic surface~\eqref{eq1.1} over~$\mathbf{F}^{\mathrm{sep}}$ is isomorphic to $\mathbb{Z}/8\mathbb{Z}$. Meanwhile, by Theorem~\ref{t1}, $(\mathrm{iii})'$ and~$(\mathrm{iv})'$, and Remark~\ref{r1} the automorphism group of the Fermat cubic surface over~$\mathbf{F}^{\mathrm{sep}}$ is isomorphic to $(\mathbb{Z}/3\mathbb{Z})^3 \rtimes \mathfrak{S}_4$. Since by Lemma~\ref{l56} the Fermat cubic surface is isomorphic to the surface~\eqref{eq1.2} over~$\mathbf{F}^{\mathrm{sep}}$, the automorphism group of the surface~\eqref{eq1.2} is also isomorphic to $(\mathbb{Z}/3\mathbb{Z})^3 \rtimes \mathfrak{S}_4$. This means that the surface~\eqref{eq1.1} is isomorphic neither to the Fermat cubic surface, nor to the surface~\eqref{eq1.2}.

By Theorem~\ref{t1},\,$(\mathrm{iv})''$, and Remark~\ref{r1} the automorphism group of the Clebsch cubic surface over a~separably closed field of characteristic different from $2$, $3$ and~$5$ is isomorphic to $\mathfrak{S}_5$, and by Lemma~\ref{l57} the automorphism group of the cubic surface~\eqref{eq1.1} over a~separably closed field is isomorphic to $\mathbb{Z}/8\mathbb{Z}$. Over a~field of characteristic~$5$ the Clebsch cubic surface is singular, while the surface~\eqref{eq1.1} is smooth. This means that the surface~\eqref{eq1.1} is not isomorphic to the Clebsch cubic surface.

The lemma is proved.
\end{proof}

\begin{proof}[of Proposition~\ref{p1}]
One can easily check assertion~(i). Assertion~(ii) follows from Lemma~\ref{l58}. Assertion~(iii) follows directly from Corollary~\ref{c6},\,(ii).

Let us prove~(iv). If $\mathbf{F}$ contains a~nontrivial cube root of unity, then the surface~\eqref{eq1.2} and the Fermat cubic surface are isomorphic by Lemma~\ref{l56}. If~$\mathbf{F}$~contains no nontrivial cube roots of unity, then the surface~\eqref{eq1.2} and the Fermat cubic surface are not isomorphic by Lemma~\ref{l55}.

We prove assertion~(v). If $\mathbf{F}$ is of characteristic~$3$, then it
follows from assertion~(i). If $\mathbf{F}$ has characteristic~$5$, then the
Clebsch cubic surface is isomorphic neither to the surface~\eqref{eq1.2}, nor
to the Fermat cubic, because in this case the \mbox{Clebsch} cubic surface is
singular, while the other surfaces are smooth. Assume that $\mathbf{F}$ is of
characteristic different from~$5$. Then by
Theorem~\ref{t1},\,$(\mathrm{iv})''$, Remark~\ref{r1} and Example~\ref{e2}
the automorphism group of the Clebsch cubic surface over
$\mathbf{F}^{\mathrm{sep}}$ is isomorphic to~$\mathfrak{S}_5$. By
Theorem~\ref{t1},\,$(\mathrm{iv})'$, Remark~\ref{r1} and Lemma~\ref{l56} the
automorphism groups of the surface~\eqref{eq1.2} and the Fermat cubic over
$\mathbf{F}^{\mathrm{sep}}$ are isomorphic to~$(\mathbb{Z}/3\mathbb{Z})^3
\rtimes \mathfrak{S}_4$. This means that over $\mathbf{F}$ the Clebsch cubic
is isomorphic neither~to the surface~\eqref{eq1.2}, nor to the Fermat cubic
surface.

Proposition~\ref{p1} is proved.
\end{proof}

\end{document}